\documentclass[journal,twocolumn]{IEEEtran}
\usepackage[cmex10]{amsmath}
\usepackage{amssymb}
\usepackage{bbm}

\usepackage{color}

\usepackage[pdftex]{graphicx}
\graphicspath{{./graph/}}
\usepackage[tight,footnotesize]{subfigure}

\newcommand{\argmax}{\operatornamewithlimits{argmax}}

\newcommand{\cG}{\mathcal{G}}   
\newcommand{\cN}{\mathcal{N}}   
\newcommand{\cL}{\mathcal{L}}   
\newcommand{\cD}{\mathcal{D}}   

\newcommand{\cM}{\mathcal{M}}
\newcommand{\cR}{\mathcal{R}}

\newcommand{\bE}{\mathbb{E}}

\newcommand{\bR}{\mathbb{R}}

\newcommand{\vx}{\boldsymbol{x}}
\newcommand{\vr}{\boldsymbol{r}}
\newcommand{\vX}{\boldsymbol{X}}
\newcommand{\vR}{\boldsymbol{R}}

\newcommand{\vnu}{\boldsymbol{\nu}}
\newcommand{\vc}{\boldsymbol{c}}
\newcommand{\vepsilon}{\boldsymbol{\epsilon}}

\newcommand{\vP}{\boldsymbol{P}}

\newcommand{\vQ}{\boldsymbol{Q}}
\newcommand{\vM}{\boldsymbol{M}}
\newcommand{\indicator}{\boldsymbol{1}}

\newcommand{\itoj}{\langle i,j\rangle}
\newcommand{\jtoi}{\langle j,i\rangle}

\newtheorem{theorem}{Theorem}
\newtheorem{proposition}{Proposition}
\newtheorem{assumption}{Assumption}
\newtheorem{definition}{Definition}
\newtheorem{lemma}{Lemma}

\IEEEoverridecommandlockouts 

\begin{document}
\author{{Haozhi~Xiong,~Ruogu~Li,~Atilla~Eryilmaz~and~Eylem~Ekici}
\thanks{Haozhi~Xiong,~Ruogu~Li,~Atilla~Eryilmaz~and~Eylem~Ekici
({\small\{xiongh,~lir,~eryilmaz,~ekici\}@ece.osu.edu}) are with the
Department of Electrical and Computer Engineering at The Ohio State
University, Columbus, Ohio 43210 USA.}}
\title{Delay-Aware Cross-Layer Design for Network Utility Maximization
 in Multi-hop Networks}
\maketitle

\begin{abstract}

We investigate the problem of designing \textit{delay-aware} joint
flow control, routing, and scheduling algorithms in general
multi-hop networks for maximizing network utilization. Since the
end-to-end delay performance has a complex dependence on the
high-order statistics of cross-layer algorithms, earlier
optimization-based design methodologies that optimize the long term
network utilization are not immediately well-suited for delay-aware
design.
This motivates us in this work to develop a novel design framework
and alternative methods that take advantage of several unexploited
design choices in the routing and the scheduling strategy spaces. In
particular, we reveal and exploit a crucial characteristic of back
pressure-type controllers that enables us to develop a novel link
rate allocation strategy that not only \emph{optimizes long-term
network utilization}, but also yields \emph{loop free multi-path
routes} between each source-destination pair. Moreover, we propose a
regulated scheduling strategy, based on a token-based service
discipline, for shaping the per-hop delay distribution to obtain
highly \emph{desirable end-to-end delay performance}. We establish
that our joint flow control, routing, and scheduling algorithm
achieves loop-free routes and optimal network utilization. Our
extensive numerical studies support our theoretical results, and
further show that our joint design leads to substantial end-to-end
delay performance improvements in multi-hop networks compared to
earlier solutions.

\end{abstract}

\section{Introduction}

Communication networks are expected to serve a variety of essential
applications that demand high long-term throughput and low
end-to-end delay.
%
%
%
Over the last decade, we have witnessed the development of
increasingly sophisticated optimization and control techniques
targeting flow control, routing and scheduling components to address
the cross-layer resource allocation problems for communication
networks.
Among various different developed policies, an important class of
\emph{throughput-optimal policies} has evolved from the seminal work
\cite{taseph92} of Tassiulas and Ephremides, where they proposed the
well-known back-pressure scheduling/routing policy. This policy
utilizes properly maintained queue-length information to dynamically
determine scheduling and routing decisions that optimize the
long-term maximal throughput levels between all source-destination
pairs. More recent works (e.g. \cite{linshr04, sto05, neemodli05,
erysri06a}; also see \cite{linshrsri06, geoneetas06, shasri07} and
references therein) extended this framework by developing an
optimization-based design methodology for the development of joint
flow control, routing, and scheduling algorithms to maximize the
long-term utilization of the network resources, measured through
proper functions of throughput.

However, existing works have dominantly concentrated only on the
long-term performance metrics of throughput or utilization, while
ignoring the end-to-end delay metric that is crucial to many
essential applications. This restriction allowed for the formulation
of the design problem as a static optimization problem in terms of
the mean behavior of the network algorithm operation. In particular,
long-term design objectives can easily be described in terms of the
mean flow rates and the mean link rates that the network algorithm
provides to the applications. Unfortunately, this approach is no
longer applicable when end-to-end delay is taken into account as
end-to-end delay is a complex stochastic process that depends on the
higher order statistics of the network algorithm operation. This
calls for a more sophisticated delay-aware cross-layer design
framework, and novel strategies that provide favorable end-to-end
delay performance, while preserving long-term optimality
characteristics.

With this motivation, in this paper, we are interested in the well-founded design of cross-layer network algorithms for general multi-hop networks that not only utilizes the network resources for long-term throughput optimality, but also exhibits desirable end-to-end delay characteristics. Our contributions in this direction can be summarized as follows:


$\bullet$
We propose a novel
design paradigm that decouples the objectives of long-term utility
maximization from the end-to-end delay-aware resource allocation.
This framework is expected to enable the systematic development of
future schemes, in addition to the one we develop in this paper.

$\bullet$ We reveal the solution given by the back-pressure policy
has a unique structure which facilitate us to establish
\emph{loop-free} multi-path routes that \emph{guarantee long-term
network utility maximization}. This new routing-scheme not only
inherits the adaptive and optimal nature of the long-term optimal
algorithms, but also eliminates the unnecessary loops to reduce the
end-to-end delay without sacrificing from throughput.

$\bullet$ We combine this loop-free route construction strategy with
a token-based scheduling discipline that regulates the higher-order
statistics of service processes to achieve drastic reductions in the
end-to-end delay performance, while guaranteeing long-term
optimality characteristics.
%


The above points also constitute the main differences of this work
from the recent efforts in the design of algorithms with low
end-to-end delay performance (e.g. \cite{yinshared09,buisristo09}).
We also note that there has been recent interest in deriving
fundamental bounds on the delay performance
(\cite{gupshr09,karluosar09,nee08}). In our future work, we are
interested in utilizing and extending these results to study the gap
between these fundamental bounds and the delay performance of our
algorithms.

The remainder of the paper is organized as follows. Section~\ref{sec:model} introduces our system model and objectives together with the description of several existing algorithms. In Section~\ref{sec:example}, we propose our delay-aware design framework that describes the desired characteristics and interconnections of the routing, scheduling, and flow control components of the network algorithm. In Section~\ref{sec:routing}, we build upon the proposed design framework to construct a delay-aware cross-layer algorithm that performs the desired tasks. The numerical result of the policy and our concluding remarks are provided in Section~\ref{sec:numerical} and \ref{sec:conclusion}, respectively.

\section{System Model and Objective}\label{sec:model}

In this work, we study wired networks to concentrate on the delay behavior of network algorithms without the additional complications of interference limited wireless communications. Discussion of possible extension of our proposed frameworks to interference limited networks is provided in the conclusions. We consider a fixed multi-hop network represented by graph $\cG=(\cN,\cL,\vc)$, where $\cN$ is the set of nodes, $\cL$ is the set of \emph{bidirectional} links $(i,j)$ where $i,j\in\cN$. We use $\vc=(c_{ij})_{\{(i,j)\in \cL\}}$ to denote the vector of \emph{bidirectional link capacities} in packets per slot, i.e., both $c_{ij}$ and $c_{ji}$ refer to the bidirectional link capacity of the bidirectional link $(i,j)$. For clarity, we use $\itoj$ to denote the \emph{directed} link from node $i$ to node $j$. Time is slotted in our system, and external packets arrive at the beginning of each time slot.


The network resources are to be shared by a set of commodities.
We distinguish different \emph{commodities} by their destinations. We define $\cD$ to be the set of all destination nodes.
In this paper, we are interested in designing joint
flow-control, scheduling, and routing policies with desirable long-term
throughput and short-term delay characteristics. Our dual goal will be
discussed in further detail after we introduce some notations.

In each time slot, the service on the link $\itoj$ of commodity $d$ is denoted by $R_{ij}^d[t]$, which is assumed to be a stationary ergodic stochastic process. It is determined by the scheduling and routing policy. We let $r_{ij}^d=\lim_{t \rightarrow\infty} \frac{1}{t} \sum_{\tau=0}^t \mathbb{E} (R_{ij}^d[\tau])$ to be the link rate of commodity $d$ on link $\itoj$, and define $\vr=(r_{ij}^d)_{i,j,d}$ to be the vector of all such link rates. Under flow control mechanism, the number of the exogenous packets that arrives at node $s$ destined to node $d$ at time slot $t$ is denoted by $X_s^d[t]$, which is also assumed to be stationary and ergodic. Similarly, $x_s^d=\lim_{t \rightarrow\infty} \frac{1}{t} \sum_{\tau=0}^t  \mathbb{E} (X_{s}^d[\tau])$ is the corresponding rate, and we let $\vx=(x_s^d)_{s,d}$ to be exogenous arrival rate vector. A utility function $U_{sd}(x_s^d)$ is associated with each source-destination pair $(s,d)$. We make the following typical assumption on the utility function:
\begin{assumption}
The utility functions $\{U_{sd}(x_s^d)\}_{s,d}$ are strictly concave, twice
differentiable and increasing functions.\label{ass:concave}\footnote{This is
not a critical assumption, but will make our analysis clear.}
\end{assumption}

\subsection{Objective}

Our goal is to develop a joint flow-control, scheduling, and routing algorithm
that not only optimizes long-term network utilization, but also provides
desirable delay characteristics. We discuss these two goals next.

\noindent \textbf{Utility Maximization: } For the network $(\cN,\cL,\vc)$, the
network utilization maximization optimization problem is defined as:
\begin{eqnarray}
&\displaystyle \max_{\{\boldsymbol{X}[t],\boldsymbol{R}[t]\}} &
  \sum_{s,d}U_{sd}(x_s^d)\label{eqn:stochasticBP}\\
  & s.t. & X_s^d[t]\geq 0,\,\forall s,d\in\cN,\,\forall t\geq 0, \label{eqn:BPcondition1}\\
  && R_{ij}^d[t]\geq 0,\,\forall (i,j)\in\cL,\,\forall t\geq 0,\label{eqn:BPcondition2}\\
  && \displaystyle\sum_{d}R_{ij}^d[t] + \sum_{d}R_{ji}^d[t]\leq c_{ij},\nonumber\\
  &&\,\forall (i,j)\in\cL,\,d\in\cD,\,\forall t\geq 0,\label{eqn:BPcondition3}\\
  && \displaystyle x_i^d+\sum_{\langle m,i\rangle\in\cL} r_{mi}^d\leq\sum_{\itoj\in\cL}
  r_{ij}^d,\nonumber\\
  &&\mbox{    }\forall i\in\cN,\,\forall d\in\cD,\,i\neq d.\label{eqn:BPcondition4}
\end{eqnarray}

We define the \emph{feasible} solution as:
\begin{definition}
(Feasible solution) A solution $(\boldsymbol{X}[t],\boldsymbol{R}[t])$ is
\emph{feasible} if it satisfies conditions (\ref{eqn:BPcondition1}) to
(\ref{eqn:BPcondition4}).\label{def:feasible} \hfill $\diamond$
\end{definition}

We use $\{\boldsymbol{X}^*[t],\boldsymbol{R}^*[t]\}$ to denote the optimal
solution to (\ref{eqn:stochasticBP}), and we define
\begin{eqnarray}
\vx^*&=&\lim_{t\rightarrow\infty}\frac{1}{t}\sum_{\tau=0}^{t-1}\boldsymbol{X}^*[\tau],\label{eqn:optimalx}\\
\vr^*&=&\lim_{t\rightarrow\infty}\frac{1}{t}\sum_{\tau=0}^{t-1}\boldsymbol{R}^*[\tau].\label{eqn:optimalr}
\end{eqnarray}
Note that under Assumption~\ref{ass:concave}, $\vx^*$ is unique while $\vr^*$
generally belongs to a set of optimal link rates denoted by $\cR^*$.

In this paper, we define the stability as follows:
\begin{definition}(Network stability)\label{def:stability}
We say the network is \textit{stable} in the mean sense if for any queue (price) $P_{nj}^d$, $\forall n,j\in\cN$, $\forall d\in\cN$, we have $\limsup_{t\rightarrow\infty}\frac{1}{t}\sum_{\tau=0}^{t-1}\bE \Bigl[P_{nj}^d[\tau]\Bigr]<\infty.\quad\diamond$
\end{definition}

\noindent \textbf{Delay Improvement: }A step beyond just solving the above
optimization, our second goal is to develop a new mechanism that reduces the
end-to-end delay experienced by the traffic while maintains the utility
maximizing nature. The end-to-end delay experienced by one packet is defined as
the difference between the time instance of injection at the source and
reception at the destination, which is a short-term metric instead of a
long-term throughput. We also aim to find a new architecture that can decouple
the different objectives.

\subsection{Background}\label{sec:background}
The back-pressure algorithm maintains a queue for each commodity $d$ in each node $i$ whose length at time $t$ is denoted by $P_i^d[t]$. The algorithm is proposed for a general network model, and specifically, the back-pressure algorithm under our system model is given by \vspace{1ex} \hrule \vspace{1ex} \textbf{The Discrete Time Back-Pressure (DTBP)
Policy}~(\cite{Neely03,erysri06a})
\begin{itemize}
\item \textbf{Queue (Price) Evolution:}\\
Each queue evolves as
\begin{equation}
\scriptsize P_i^d[t+1] = \left(P_i^d[t]- \sum_{\langle
i,j\rangle\in\cL}R_{ij}^d[t]\right)^+ +X_i^d[t]+\sum_{\langle m,i\rangle\in\cL}
R_{mi}^d[t]\label{eqn:priceEvolution}
\end{equation}
\item \textbf{Rate control:}\\
In time slot $t$, the flow controller determines the mean of the injection as $$\mathbb{E}(X_s^d[t]\mid P_s^d[t]) = \min\{U_{sd}^{'-1}(P_s^d[t]/K),X_{max}\},$$
where $X_{max}$ is some finite upper bound for the arrival rate.
\item \textbf{Scheduling and Routing:}\\
Each link $l=(i,j)$ selects
\begin{equation}
d_l^*[t]=\argmax_{d}
\left|P_i^d[t]-P_j^d[t]\right|,\,l\in\cL,\label{eqn:dtbprouting}
\end{equation}
and assign
\begin{eqnarray}
\scriptsize R_{ij}^d[t] = \left\{
\begin{array}{ll}
    c_{ij},& \hbox{if }d=d_l^*[t]\,\hbox{and}\,
    P_i^{d}[t]-P_j^{d}[t]>0;\\
    0,& \hbox{otherwise}.
\end{array}\right.\label{eqn:rateAssign}
\end{eqnarray} In the event of multiple commodities satisfying
Equation~(\ref{eqn:dtbprouting}), we arbitrarily choose one such commodity with equal probability as a \emph{tie-breaking rule}.

\end{itemize}
\vspace{1ex} \hrule \vspace{1ex}

The constant $K$ in the algorithm is a design parameter that determines how
close the algorithm can converge to its optimal solution (see
\cite{Neely03,erysri06a}). Also we shall see in Section~\ref{sec:routing} this solution
is related to dual decomposition.



\emph{Remark 1: }In our model, since there is no interference between links, the back-pressure algorithm reduces to making decisions on each link independently. Since each link can only transmit in one direction at a given time slot, it is necessary to use the weight defined in Equation~(\ref{eqn:dtbprouting}) to determine which one of the directional links should transmit.

\emph{Remark 2: }Note that the rate assignment in Equation (\ref{eqn:rateAssign}) implies that we only allow transmission on those links with \emph{strictly positive} maximum back-pressure. This does not affect the optimality of the DTBP algorithm.

\emph{Remark 3: }The price evolution in Equation (\ref{eqn:priceEvolution}) implies that in the case of a node has less packets than scheduled service, dummy packets are transmitted. Thus the allocated service $R_{ij}^d[t]$ always \emph{equals} to number of packets transmitted over link $\itoj$ at time $t$ as suggested in Equation (\ref{eqn:priceEvolution}).

There are many follow up works of the back-pressure algorithm. In particular, in \cite{buisristo09}, an interesting \emph{min-resource algorithm} is proposed based on the back-pressure algorithm. In this implementation, instead of using the queue-length difference $\left|P_i^d[t]-P_j^d[t]\right|$ as the weight of a link, $\left(\left|P_i^d[t]-P_j^d[t]\right|-M\right)$ is used, where $M$ is some positive constant. This modification discourages the use of links unless the queue-length differences are sufficiently high, and hence reduces possible loops in the network.

\section{Delay-Aware Design Framework}\label{sec:example}
In this section, we expose the delay deficiencies of long-term utility
maximizing designs such as DTBP and its variants, both conceptually and through
numerical studies. In particular, we reveal that both the multi-path routing
and the scheduling components of the earlier designs must be significantly
changed or enhanced to obtain favorable end-to-end delay performance. Based on
the observations, we propose a general delay-aware design framework whereby the
utility maximization is combined with delay-aware routing and scheduling
components to achieve our dual objective.

%
%
\subsection{Deficiencies in multi-path routing}\label{sec:loop}

Long-term performance optimizing algorithms share the common characteristic of
continuously searching for routes to maximize the end-to-end throughput of the
flows. However, this potentially causes loops and unnecessarily long routes for
a large subset of the packets, as we shall demonstrate in the following
example.

Consider a $6\times6$ grid with unit capacity links serving two commodities as shown in Fig.~\ref{fig:backpressurepath}. The source-destination node pair for Commodity $1$ and Commodity $2$ are marked with square and circle respectively. Under the operation of DTBP, the hop-count distribution of Commodity $1$ packets converges to the plotted distribution. We can observe that the hop count exceeds the maximum possible loop-free path length of $35$ for a non-negligible fraction of the packets. It turns out that such behavior is typical for similar long-term optimal policies under different setups. In this work, we are interested in preserving the long-term optimality of such solutions, but also eliminating loops and hence significantly improving the delay performance. We will explore the diversity in the set of optimal link rates $\cR^*$ to achieve this.

%
%
\begin{figure}[!ht]
\centering
\includegraphics[width=2.5in]{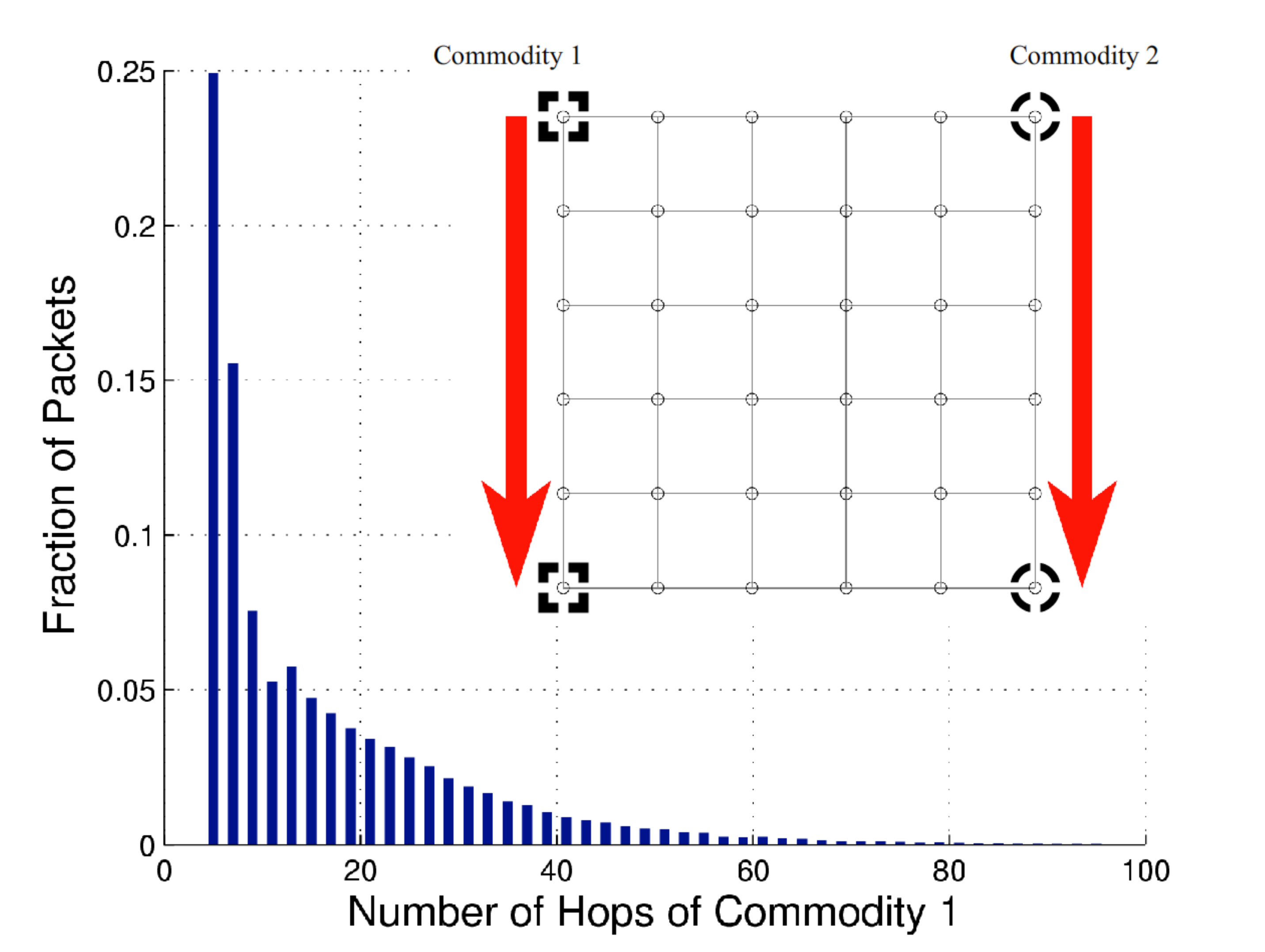}
\caption{The distribution of hop-count for flow 1 packets indicates the
presence of unnecessary loops in their routes} \label{fig:backpressurepath}
\end{figure}

\subsection{Deficiencies in scheduling and flow control}\label{sec:exampleservice}

Long-term utility maximizing policies such as DTBP converge to a set of mean
flow and link rates that solves (\ref{eqn:stochasticBP}). Yet, the end-to-end
delay of these policies is a complex function of higher order moments of the
packet arrival and service processes that are determined by the flow controller
and the scheduler.

As a simple example, it is well-known~(\cite{kle75b}) that in a G/G/1 queueing
system, the mean waiting time $W$ is bounded by
%
%
%
\begin{equation}
W\leq\frac{\sigma_a^2+\sigma_b^2}{2\bar{t}(1-\rho)},\label{eqn:delaybound}
\end{equation}
where \(\sigma_a^2\) is the variance of inter-arrival time, \(\sigma_b^2\) is
the variance of service time, \(\bar{t}\) is the mean of inter-arrival time,
and \(\rho\) is the utility factor. This formulation suggests that reducing the
variance of the arrival and service processes is crucial in reducing the delay.

In a large scale multi-hop network, it is not possible to formulate the
end-to-end delay as a function of the arrival and service processes. However,
motivated by the above observation, in this work, we are interested in
regulating the flow controller and scheduling components of the long-term
optimal design to improve delay performance.

%

\subsection{Proposed multi-layer design framework}\label{sec:architecture}

The previous two subsections expose deficiencies in the routing, scheduling,
and flow control components of existing long-term utility maximizing
algorithms, and reveal new opportunities in the design of delay-aware cross
layer algorithms. These motivate us in this section, to propose a novel
\emph{multi-layer design framework} for multi-purpose policy design.

\begin{figure}[!ht]
\centering
\includegraphics[width=2.5in]{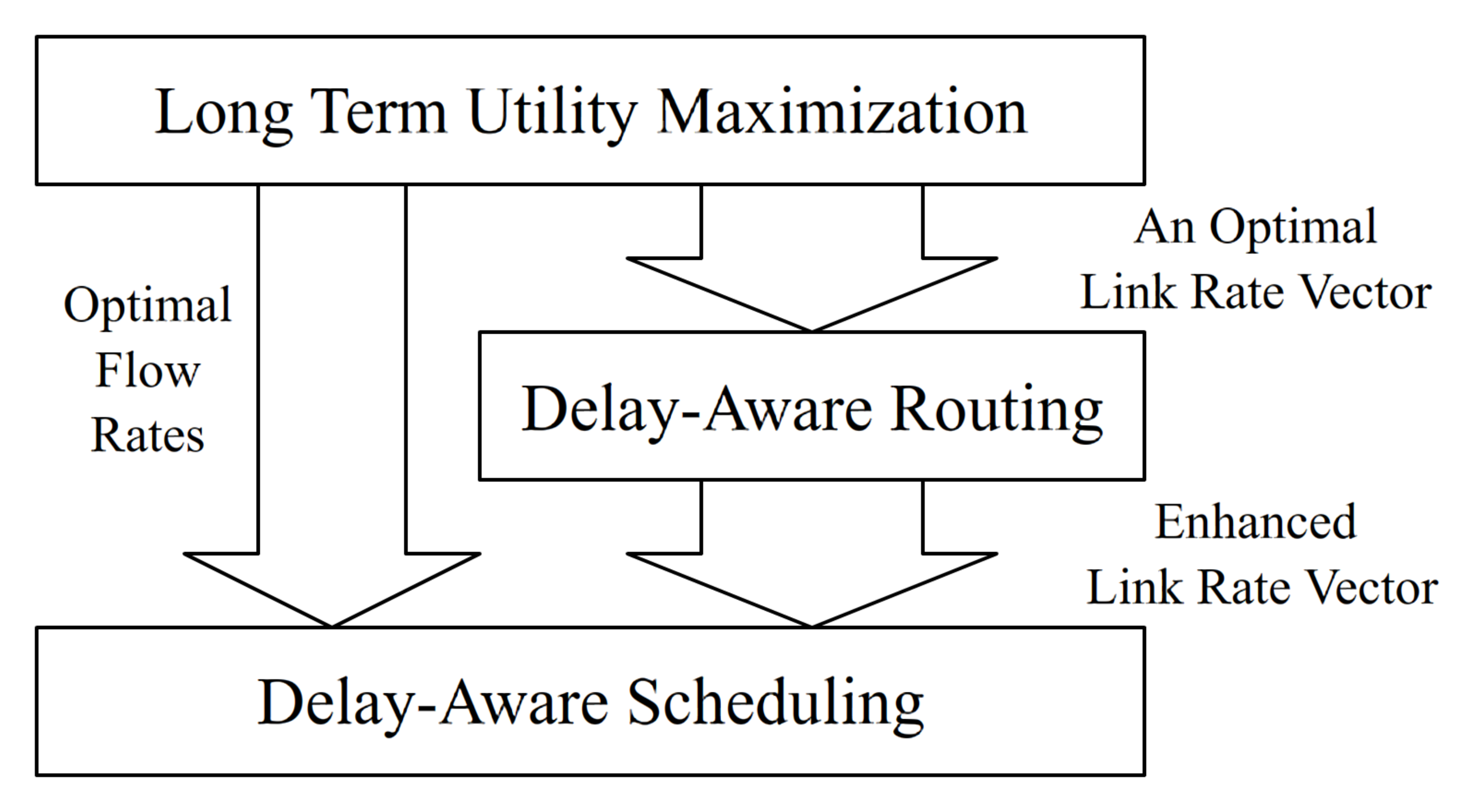}
\caption{Multi-Layer Design Framework} \label{fig:architecture}
\end{figure}

The qualitative structure of our proposed framework is shown in~Fig.~\ref{fig:architecture}, where the goals of long-term optimality, loop-free routing, and regulated flow control and scheduling are decoupled into different layers of the conceptual framework.

In the first layer, we solve the problem of utility maximization. Note that
under Assumption \ref{ass:concave}, the optimal solution for flow rates,
$\vx^*$ is unique, while $\vr^*$ may have multiple solutions. The delay-aware
routing layer exploits the freedom amongst these link-rate solutions and
obtains a set of more desirable link rates that improves delay performance
while preserving the utility maximization. Acquiring the utility maximization
flow rates form the first layer and the delay-aware link-rates form the second
layer, the delay-aware scheduler regulates the arrival and service process to
achieve better delay performance. As such, this establishes a design framework
for developing delay-aware cross-layer algorithms that possess the above
characteristics and hence is expected to be useful for algorithm designs for
diverse network scenarios. In particular, we shall next present specific
algorithms with attractive delay characteristics.


%
%

\section{Delay-aware routing and scheduling}\label{sec:routing}
In this section, we present our delay-aware routing and delay-aware scheduling algorithms for our proposed framework. Our delay-aware routing algorithm preserves utility maximizing properties while avoiding loops. Our delay-aware scheduling algorithm implements a network stabilizing token-based scheme that significantly reduces queue lengths for any feasible rate assignment. We then combine these two schemes in our delay-aware, cross layer congestion control, routing and scheduling policy following our proposed architecture of
Section~\ref{sec:architecture}.

\subsection{Delay-aware route construction}\label{sec:darouting}
To overcome the disadvantages observed in the back-pressure algorithm, a natural idea is to restrict the direction in which packets can be transmitted over a link for a certain commodity. To achieve this, we first introduce a fluid approximation associated with (\ref{eqn:stochasticBP}) for the following discussion.

\begin{eqnarray}
  &\displaystyle \max_{\vx,\vr}& \sum_{s,d}U_{sd}(x_s^d)\label{eqn:objectproblem}\\
  &s.t. & x_s^d\geq 0,\,\forall s,d\in\cN,\label{eqn:xpositive}\\
  && r_{ij}^d\geq 0,\,\forall (i,j)\in\cL,\label{eqn:rpositive}\\
  && \displaystyle\sum_{d}r_{ij}^d + \sum_{d}r_{ji}^d\leq c_{ij},\,\forall
  (i,j)\in\cL,\,d\in\cD,
  \label{eqn:capacityconstraint}\\
  && \displaystyle x_i^d+\sum_{\langle m,i\rangle\in\cL} r_{mi}^d\leq\sum_{\itoj\in\cL}
  r_{ij}^d,\nonumber\\
  && \mbox{    }\forall i\in\cN,\,\forall d\in\cD,\,i\neq d.\label{eqn:flowconservation}
\end{eqnarray}

Note that if an optimal solution of (\ref{eqn:stochasticBP}) is $\{\boldsymbol{X}^*[t],\boldsymbol{R}^*[t]\}$, it is shown in previous works (e.g. \cite{Neely03,erysri06a}) that the associated $(\vx^*,\vr^*)$ defined in (\ref{eqn:optimalx}) and (\ref{eqn:optimalr}) is an optimal solution of (\ref{eqn:objectproblem}) \footnote{More precisely, $(\vx^*,\vr^*)$ is within $O(1/K)$ of the optimal solution of (\ref{eqn:objectproblem}). Nevertheless, we still call $(\vx^*,\vr^*)$ an optimal solution of (\ref{eqn:objectproblem}) since the converge can be precise if a diminishing step-size (in our notation, if $K$ grows linearly with time $t$) is used.}.

\textbf{An Alternate Optimal Solution: }
To ensure the optimality of the proposed algorithm, we explore the solution space for an alternate optimal solution with a particular structure. First, we define the following mapping of the link rates:
\begin{definition}(Delay-aware link rate mapping) The delay-aware link rate
mapping $\hat{R}_{ij}^d[t]$ on the directed link $\itoj$ for commodity $d$ at
time $t$ is defined as
\begin{eqnarray}
\hat{R}_{ij}^{d}[t] = \frac{1}{t} \left(\sum_{\tau=0}^{t-1}\left(
R_{ij}^{d}[\tau]-R_{ji}^{d}[\tau]\right)\right)^+, \label{eqn:mapping}
\end{eqnarray}
where \(d\in\cN\), \((i,j)\in\cL\) and \((Z)^+=\max(0,Z)\). Note that this
measures the running average of the \emph{net rate} of commodity $d$ traffic
traversing link $\itoj$.\hfill$\diamond$
\end{definition}


Suppose we use $R_{ij}^d[t]=R_{ij}^{*d}[t]$ in (\ref{eqn:mapping}). Let
$\hat{r}_{ij}^{*d}=\lim_{t\rightarrow\infty}\hat{R}_{ij}^{d}[t]$ and $\hat{\vr}^*=(\hat{\vr}_{ij}^{*d})_{i,j,d}$ then $\hat{
\vr}^*$ is related to $\vr^*$ as:
\begin{eqnarray}
\hat{r}_{ij}^{*d} &=& \left(r_{ij}^{*d}-r_{ji}^{*d}\right)^+,
\label{eqn:r_ij_hat}\\
\hat{r}_{ji}^{*d} &=& \left(r_{ji}^{*d}-r_{ij}^{*d}\right)^+
\label{eqn:r_ji_hat}.
\end{eqnarray}


We shall show that the link rate mapping (\ref{eqn:r_ij_hat}) and
(\ref{eqn:r_ji_hat}) preserves flow rate optimality.

\begin{proposition}
If $(\vx^*,\vr^*)$ is an optimal solution to (\ref{eqn:objectproblem}), then $(\hat{\vx}^*,\hat{\vr}^*)$ with $\hat{\vx}^*=\vx^*$ and $\hat{\vr}^*$ being as defined in Equations (\ref{eqn:r_ij_hat}) and (\ref{eqn:r_ji_hat}) is also an optimal solution.\label{thm:optimal}
\end{proposition}
\begin{proof}
The detailed proof is provided in Appendix~\ref{sec:optimal}.
\end{proof}

\emph{Remark: } Note that under this link rate assignment, for each $(i,j)\in\cL$, at least one of $\hat{r}_{ij}^{*d}$ and $\hat{r}_{ji}^{*d}$ equals to zero as if the corresponding bidirectional link becomes a unidirectional link. Clearly, those links with equal rates on both directions will have a net rate of $0$. Also, note that this link rate mapping preserves the optimality of $(\hat{\vx}^*,\hat{\vr}^*)$ when applied to \emph{any} optimal solution $(\vx^*,\vr^*)$ to problem (\ref{eqn:stochasticBP}) which is not necessarily given by the DTBP algorithm.


\textbf{The Steady-State Behavior of the DTBP Algorithm: }
Previous works (e.g. \cite{Neely03}) also show the stability of the DTBP algorithm. In other words, the Markov chain with the queue-lengths $\vP=(P_i^d)_{i,d}$ being its states is positive recurrent and ergodic under the DTBP algorithm. We use $\pi(\vP)$ to denote the steady state distribution under the DTBP algorithm, i.e., the probability of the queue-lengths being $\vP$ is $\pi(\vP)$ after the convergence of the DTBP algorithm and $\sum_{\vP}\pi(\vP)=1$. Also, for queue $P_i^d$, the following holds:
\begin{equation}
\lim_{t\rightarrow\infty} \frac{1}{t} \sum_{\tau=0}^t  \mathbb{E} (P_{i}^d[\tau])=\sum_{\vP}P_i^d\pi(\vP),
\end{equation}
and we use
\begin{equation}
\bar{P}_i^{*d}=\sum_{\vP}P_i^d\pi(\vP)\label{eqn:averageQueue}
\end{equation} to denote the optimal average queue-length of queue $P_i^d$.

Intuitively, the link rate assignment equation (\ref{eqn:rateAssign}) of the DTBP algorithm implies that if there is a positive net flow rate from node $i$ to node $j$ for destination $d$, then the average queue-lengths should be \emph{strictly decreasing}, i.e., $\bar{P}_i^{*d}>\bar{P}_j^{*d}$, since $i$ can send a packet to $j$ only when $P_i^d[t]>P_j^d[t]$. In certain simple topologies, this statement can be proven rigorously.
For a tandem network (a network whose adjacency matrix has non-zero elements only on the first diagonal above the main diagonal and the first diagonal below the main diagonal, see Fig.~\ref{fig:tandemNetwork} for an example) with unit link capacities and one end-to-end flow, we have the following results:
\begin{lemma}
In such a tandem network with no more than 3 hops, the average queue-length of the DTBP algorithm is strictly decreasing from the source to the destination for any arrival rate $a\in(0,1)$ under a Bernoulli arrival process.\label{lmm:3hops}
\end{lemma}
\begin{proof}
The detailed proof is provided in Appendix \ref{sec:3hops}.
\end{proof}

\begin{figure}[!h]
\begin{center}
\includegraphics[width=2.5in]{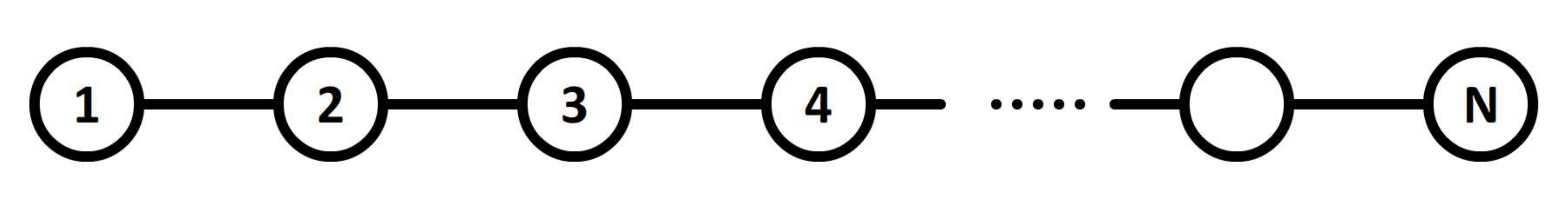}
\caption{A tandem network with $N-1$ hops.} \label{fig:tandemNetwork}
\end{center}
\end{figure}


The method used in the proof of Lemma \ref{lmm:3hops} can be extended to a $k$-hop tandem network with a Bernoulli arrival. However the analysis becomes cumbersome as $k$ increases. Thus we take a different approach to extend the above result to a more general $k$-hop tandem network.

\begin{lemma}
In a tandem network under our model with unit link capacities and one end-to-end flow with a Bernoulli arrival with rate $a\in(\frac{1}{2},1)$, the average queue-length $\bar{P}_i^{*d}$ is strictly decreasing from the source to the destination under DTBP.\label{lmm:linearNetwork}
\end{lemma}
\begin{proof}
The detailed proof is provided in Appendix \ref{sec:linearNetwork}.
\end{proof}

Due to the complex interactions in the dynamics of the queue-length in the neighboring nodes in a general network, a rigorous proof of a generalization of Lemma \ref{lmm:linearNetwork} remains an open research problem. However it can be observed from numerical studies of more general networks that, under the DTBP algorithm, the average queue-lengths are strictly decreasing over links with a positive net flow rate. Fig.\ref{fig:netRateAndQueueGrid} shows a typical example in a $6\times 6$ grid network. The arrows in the figure indicate the direction of the net flow, and the average queue-length of the flow in each node is shown by the number besides the node.

\begin{figure}[!h]
\begin{center}
\includegraphics[width=2.5in]{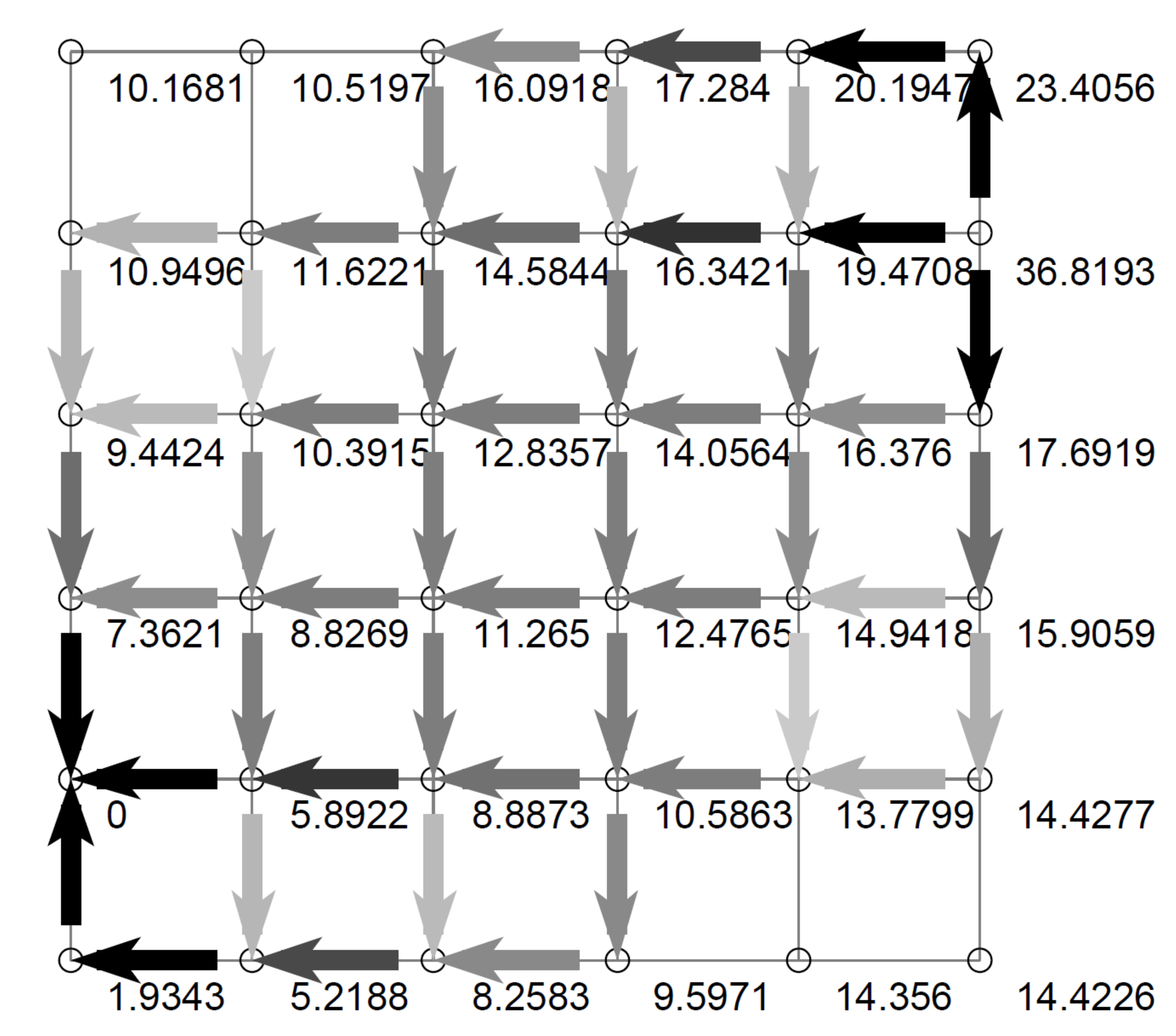}
\caption{The net flow rate and the average queue-length in a $6\times6$ grid network. The average queue-lengths are strictly decreasing in the direction of the net flow.} \label{fig:netRateAndQueueGrid}
\end{center}
\end{figure}

With such results and observations, we present the following assumption on the DTBP algorithm:

\begin{assumption}
Under our system model, the DTBP algorithm converges to a optimal solution of (\ref{eqn:stochasticBP}) $\{\boldsymbol{X}^*[t],\boldsymbol{R}^*[t]\}$ and the corresponding queue evolution $\vP^*[t]$ as defined in (\ref{eqn:priceEvolution}). The optimal solutions satisfy the following property: $$\mbox{If } r_{ij}^{*d}>r_{ji}^{*d} \mbox{ then } \bar{P}_i^{*d}>\bar{P}_j^{*d},\,\forall (i,j)\in\cL,\,d\in\cD,$$ where $r_{ij}^{*d}$ and $r_{ji}^{*d}$ are defined as in (\ref{eqn:optimalr}), and $\bar{P}_i^{*d}$ and $\bar{P}_j^{*d}$ are defined as in (\ref{eqn:averageQueue}).\label{ass:strictlyDecreasing}
\end{assumption}

\emph{Remark: }Note that the converse of Assumption \ref{ass:strictlyDecreasing} is not necessarily true, i.e., $\bar{P}_i^{*d}>\bar{P}_j^{*d}$ does not necessarily imply that $r_{ij}^{*d}>r_{ji}^{*d}$, which can be disproved by counter examples.


\textbf{Loop-Free Route Construction: }
Consider the network $\cG=(\cN,\cL,\vc)$. For each commodity $d$, we define a subgraph $\hat{\cG}^d=(\hat{\cN}^d,\hat{\cL}^d,\hat{\vc}^d)$ as:
\begin{eqnarray*}
\hat{\cN}^d &=& \cN;\\
\hat{\cL}^d &=& \cL\setminus\{\itoj\in\cL:\hat{r}_{ij}^{*d}=0\};\\
\hat{c}_{ij}^d &=& \left\{
 \begin{array}{ll}
 0 & \mbox{if }\itoj\notin\hat{\cL}^d\mbox{ and }\jtoi\notin\hat{\cL}^d;\\
 c_{ij} & \mbox{otherwise}.
 \end{array}
\right.
\end{eqnarray*}
This $\hat{\cG}^d$ is the restricted network topology that commodity $d$ sees after applying the delay-aware link rate assignment.

Then we have the following proposition:
\begin{proposition}
Under Assumption \ref{ass:strictlyDecreasing}, given a network $\cG=(\cN,\cL,\vc)$ and its optimal solution $(\vx^*,\vr^*)$ given by the back-pressure algorithm to problem (\ref{eqn:objectproblem}), the subgraph \(\hat{\cG}^d\) defined as above is loop-free $\forall\,d\in\cD$.
\label{thm:loopfree}
\end{proposition}
\begin{proof}
Assume \(\hat{\cG}^d\) has a cycle that is composed of links \(\{(n_{c_1},n_{c_2}),(n_{c_2},n_{c_3}), \cdots,(n_{c_{k-1}},n_{c_k}),(n_{c_k},n_{c_1})\} \subseteq \hat{\cL}^d\). By the definition of subgraph $\hat{\cG}^d$, we can assume without loss of generality that
\(\hat{r}_{c_ic_{i+1}}^{*d}>0\) for \(0\leq i\leq k-1\) and \(\hat{r}_{c_kc_1}^{*d}>0\). Then by Assumption \ref{ass:strictlyDecreasing}, we have $$\bar{P}_{c_1}^{*d}>\bar{P}_{c_2}^{*d}>\cdots >\bar{P}_{c_k}^{*d} >\bar{P}_{c_1}^{*d},$$ which is impossible. Therefore \(\hat{\cG}^d\) is loop-free.
\end{proof}
\emph{Remark: }As mentioned in the proposition itself, the loop-freeness holds only when the delay-aware link rate assignment is applied to the optimal solution \emph{given by} DTBP.

\textbf{Example: }
We use the topology shown in Fig. \ref{fig:triangleExample}(a) as an example to illustrate how the delay-aware link rate assignment works with the optimal solution given by DTBP. Each link in the network is assumed to have unit capacity. Node 1 is the source and Node 3 is the destination.

\begin{figure}[!h]
\begin{center}
\includegraphics[width=3in]{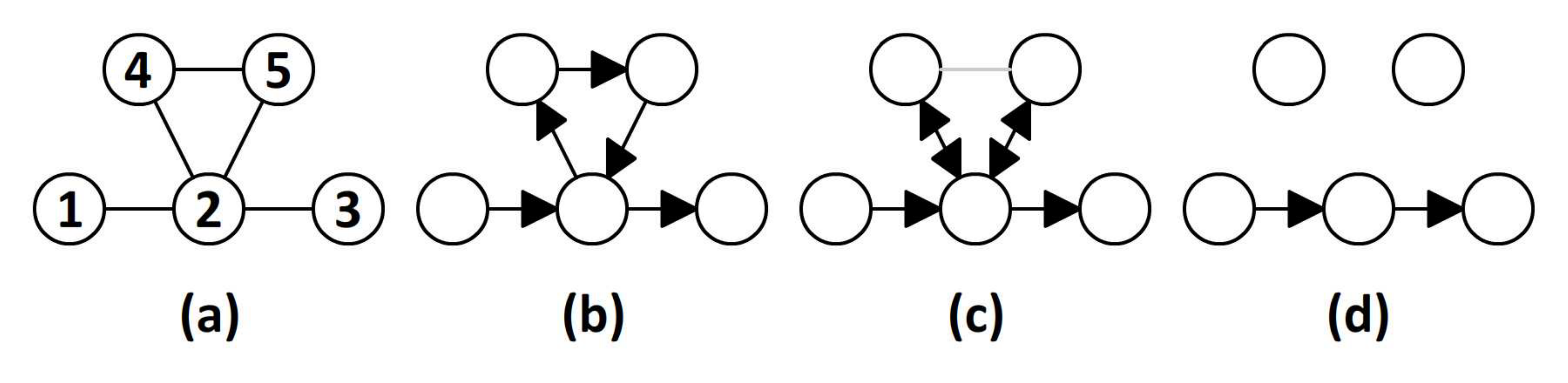}
\caption{An example to illustrate the DTBP solution with delay-aware link rate assignment.} \label{fig:triangleExample}
\end{center}
\end{figure}

Shown in Fig. \ref{fig:triangleExample}(b) is a possible optimal solution of the problem, with $r_{12}=r_{23}=1$ and $r_{24}=r_{45}=r_{52}=r$, where $r\in(0,1]$. Note that there exists a loop $(2\rightarrow4\rightarrow5\rightarrow2)$ in this optimal solution. However, analyzing the corresponding Markov chain shows that the DTBP algorithm can never converge to such an optimal solution. In fact, the DTBP converges to a solution shown in Fig. \ref{fig:triangleExample}(c), where $r_{12}=r_{23}=1$, link $(2,4)$ and $(2,5)$ have equal and non-zero transmission rate in either direction, and there is no transmission on link $(4,5)$. Note that this confirms Assumption \ref{ass:strictlyDecreasing} in this particular topology, and also confirms that the optimal solutions given by DTBP is a subset of all optimal solutions with a special structure.

If we were to apply the delay-aware link rate assignment to the optimal solutions shown in Fig. \ref{fig:triangleExample}(b), it can be verified that the loop $(2\rightarrow4\rightarrow5\rightarrow2)$ still exists in the resulting topology. On the other hand, Fig. \ref{fig:triangleExample}(d) shows the resulting topology when the delay-aware link rate assignment being applied to the optimal solution given by DTBP algorithm, which is a loop-free route.

\subsection{Delay-aware scheduler design}\label{sec:regulation}
Using the delay-aware routing of Section~\ref{sec:darouting}, utility maximization can be achieved while avoiding loops, which leads to lower average end-to-end delays. Here, we seek to improve the delay performance further through delay-aware scheduling at each link. As argued in Section~\ref{sec:exampleservice}, significant per hop delay improvement can be achieved by reshaping the service distribution through regulating the scheduler operation. In previous works~\cite{wusri05}, the notion of regulators were introduced to help stabilize the network with fixed single-path routing. Adopting a similar approach, we propose a delay-aware scheduler with regulated service that stabilizes $(\cN,\cL,\vc)$ when there exists a loop-free $\vR[t]$ such that $\vX[t]$ satisfies $\vx\in int\bigl(\Lambda(\cN,\cL,\vc)\bigr)$. We assume that the arrivals take place after the service in each time slot, i.e., the arrivals cannot be served in the arriving time slot. We further assume that from now on all links have the same link capacity denoted by $c$. 

\begin{figure}[!ht]
\centering
\input{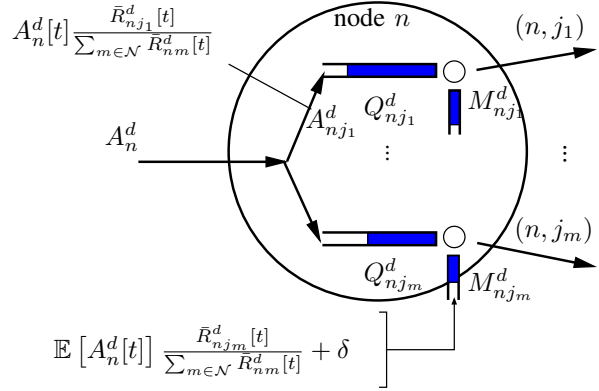}
\caption{Delay-Aware Scheduler Structure}\label{fig:token}
\end{figure}


\begin{definition}(Delay-aware scheduler with regulated service)
The delay-aware scheduler consists of \emph{per-neighbor-per-commodity} queues, i.e., the scheduler at node $n$ maintains queues $Q_{nj}^d$ for each next-hop neighbor $j$ and each commodity $d$. The servers apply regulated service discipline to shape the traffic. Let the service and departure of the server be $S(t)$ and $D(t)$, respectively. 

Let $A_n^d[t]$ denote the commodity $d$ packets entering node $n$ in slot $t$. Mathematically,
\begin{equation*}
A_n^d[t]=X_n^d[t]\mathbbm{1}_{\{n\in\cN_{source}^d\}}+\sum_{k\in\cN}D_{kn}^d[t],
\end{equation*}
where $\cN_{source}^d$ denotes the source node set with respect to commodity $d$. Let $A_{nj}^d[t]$ denote the number of packets of commodity $d$ arriving at queue $Q_{nj}^d$ in slot $t$ and $A_n^d[t]=\sum_{j\in\cN}A_{nj}^d[t]$. The queue dynamics for each queue is given as:
\begin{eqnarray}
Q_{nj}^d[t+1]&=&\Bigl(Q_{nj}^d[t]-D_{r,nj}^d[t]\Bigr)^++A_{nj}^d[t],\nonumber\\
&&\qquad\qquad\qquad\qquad\forall n\neq d,\label{eqn:dynamics1}
\end{eqnarray}

The server of the delay-aware scheduler operates with a token-based service discipline which consists of three components:
\begin{itemize}
\vspace{0.2cm}
\item \textbf{Arrival Splitting:}\\
Each commodity $d$ packet arriving at node $n$ goes into the per-neighbor-per-commodity queue $Q_{nj}^d$ with probability $\frac{\bar{R}_{nj}^{d}[t]}{\sum_{m\in\cN}\bar{R}_{nm}^{d}[t]}$, resulting in
\begin{equation}
\mathbb{E}\left[A_{nj}^d[t]\Bigm\vert A_n^d[t]\right]=A_n^d[t]\frac{\bar{R}_{nj}^{d}[t]}{\sum_{m\in\cN}\bar{R}_{nm}^{d}[t]},\label{eqn:arrival}
\end{equation}
where $\bar{R}_{nj}^{d}[t]=\frac{1}{t}\sum_{\tau=0}^{t-1}R_{nj}^d[\tau]$.
This splitting preserves the same mean link rates as the original $\vR[t]$. As an example, the arrival splitting can be done using a token-based scheme.
\vspace{0.2cm}
\item \textbf{Token Generation:}\\
For each link $\itoj\in\cL$, the system maintains \emph{token counters} $m_{nj}^d$ for each regulator queue $Q_{nj}^d$ as shown in Fig.~\ref{fig:token}. As the input to the delay-aware scheduler is loop-free, the transmissions can only be unidirectional for certain commodity $d$ over link $(i,j)$. Hence $m_{nj}^d$ and $m_{jn}^d$ cannot be non-zero simultaneously. Then we define $M_{nj}^d=m_{nj}^d+m_{jn}^d$. The token arrives at each counter with rate
\begin{equation}
S_{nj}^d=\mathbb{E}\left[A_n^d[t]\right]\frac{\bar{R}_{nj}^{d}[t]}{\sum_{m\in\cN}\bar{R}_{nm}^{d}[t]}+\delta,\label{eqn:regulator}
\end{equation}
where $\delta>0$. Note that the token arrival rates are enlarged by $\delta$ which is allowed by the excess capacity from the utility maximizing solution for $(\cN,\cL,\vc-\vepsilon)$.

\vspace{0.2cm}
\item \textbf{Service:}\\
In each time slot $t$, the token-based server chooses a winner commodity $d_{nj}^*$ with a non-zero backlog for each link $(n,j)$ as follows:
\begin{equation}
d_{nj}^*[t]=
\begin{cases}
\argmax_d\left(M_{nj}^d[t]-c_{}\right), \\
\qquad\quad\text{ if }\max_d\left[M_{nj}^d[t]-c_{}\right]>0;\\
\varnothing, \qquad\text{otherwise}.
\end{cases}
\label{eqn:tokenchoose}
\end{equation}
The server serves the commodity with the largest token count that exceeds the link capacity $c_{}$. Note that the link capacity $c_{}$ is shared by transmissions on link (n,j) in both directions. Then
\begin{equation}
D_{nj}^d[t]=
\begin{cases}
c_{}, \qquad&\text{if } d=d_{nj}^*;\\
0, \qquad&\text{if }d\neq d_{nj}^*.
\end{cases}
\end{equation}
In case of insufficient packets in queue, the server sends out dummy packets to ensure a departure of $c$ packets. After service, the token counter $M_{nj}^{d_{nj}^*}$ for the commodity $d_{nj}^*$ is decreased by $c$.\hfill $\diamond$
\end{itemize}
\end{definition}

The above algorithm does not require the knowledge of $\vR[t]$ as long as $\bar{\vR}[t]:=(\bar{R}_{nj}^d[t])_{n,j,d}$ is provided.
\begin{definition}(Token count process)
We call $M[t]$ a \textit{token count process} if:
\begin{equation*}
M[t+1]=
\begin{cases}
M[t]+\nu, &\text{if \(M[t]+\nu<c^{th}\);}\\
M[t]+\nu-c^{th}, &\text{if \(M[t]+\nu\geq c^{th}\).}
\end{cases}
\end{equation*}
\begin{equation*}
Z[t+1]=
\begin{cases}
0, &\text{if } M[t]+\nu<c^{th};\\
c^{th}, &\text{if } M[t]+\nu\geq c^{th}.
\end{cases}
\end{equation*}
where $Z[t]$ is the associated \textit{token process}, \(0<\nu<c^{th}\) is the \textit{token generation rate} and $c^{th}$ is the decision threshold. \hfill $\diamond$
\end{definition}
Following our definitions of delay-aware scheduler, it is clear that the service process of each scheduler is a token count process. We have the following lemma showing important properties of the token count process, that is, given a finite initial state, the token count process evolves towards a bounded invariant region.
\begin{lemma}\label{thm:attinv}
Let $\vM[t]=(M_1[t],\ldots,M_n[t])$ be the token count vector with token generation rate $\vnu=(\nu_1,\ldots,\nu_n)$ and decision threshold $c^{th}$. Let $D_i[t]$ be the departure from token counter $i$ at time $t$, $\forall 1\leq i\leq n$. The delay-aware scheduler serves $\vM[t]$ with the following service discipline:
\begin{eqnarray*}
i^*[t]&=&
\begin{cases}
\argmax_i\left(M_{i}[t]-c^{th}\right), \\
\qquad\quad\text{ if }\max_i\left[M_{i}[t]-c^{th}\right]>0;\\
\varnothing, \qquad\text{otherwise}.
\end{cases}\\
D_{i^*}[t]&=&
\begin{cases}
c^{th}, \qquad&\text{if } i=i^*;\\
0, \qquad&\text{if }i\neq i^*.
\end{cases}
\end{eqnarray*}
The token count dynamics is as follows:
\begin{equation*}
M_i[t+1]=M_i[t]-D_i[t]+\nu_i,\quad\forall 1\leq i\leq n.
\end{equation*}
Let $\cM=\bigl\{\vM\in\bR^n:\;\sum_{i=1}^n M_i<(n+1)c^{th}\bigr\}$. Suppose $\sum_{i=1}^n\nu_i<c^{th}$, then $\forall t_0>0$ such that $\forall \vM[t_0]\in\bR^n$, there exists $T<\infty$, such that $\vM[t_0+T]\in\cM$. Moreover, $\vM[t]\in\cM$, $\forall t\geq t_0+T$.
\end{lemma}
\begin{proof}
The detailed proof is provided in Appendix~\ref{sec:attinv}.
\end{proof}
Intuitively, Lemma~\ref{thm:attinv} reveals the fact that the token count process keeps track of the assigned token generation rate and the deviation is always bounded. Lemma~\ref{thm:attinv} will help us in proving the network stability when the delay-aware schedulers are used. Next we define the network capacity region for multi-hop networks.
\begin{definition}(Network capacity
  region)\label{def:networkcapacityregion} The \textit{network
    capacity region} is $\Lambda(\cN,\cL,\vc)=Co\{\vx\}$, where $\vx$
  is the exogenous arrival rate vector such that $\vx\succeq 0$ and
  such that there exists a link rate vector $\vr$ satisfying
\begin{eqnarray*}
  \vr\succeq 0,\\
  r_{nj}^d=0,\;\forall (n,j)\notin\cL,&\\
  r_{nn}^d=0,\;\forall n\in\cN,&\\
  r_{nj}^d=0,\;\forall n=d,&\\
  x_n^d\leq\sum_{j}r_{nj}^d-\sum_{n}r_{mn}^d,\;\forall n\neq d,&\\
  \sum_dr_{nj}^d+\sum_dr_{jn}^d\leq c_{nj},\;\forall (n,j)\in\cL.&\diamond
\end{eqnarray*}
\end{definition}
Notice that the solution $(\vX[t],\vR[t])$ is \textit{feasible} if
$(\vx,\vr)$ satisfies Definition~\ref{def:networkcapacityregion} such
that $\vx\in\Lambda(\cN,\cL,\vc)$, where
$\vx=\lim_{t\rightarrow\infty}\frac{1}{t}\sum_{\tau=0}^{t-1}\boldsymbol{X}[\tau]$,
$\vr=\lim_{t\rightarrow\infty}\frac{1}{t}\sum_{\tau=0}^{t-1}\boldsymbol{R}[\tau]$.

Based on the defined network capacity region, we now introduce the following proposition:
\begin{proposition}
The delay-aware scheduler operating in the network $(\cN,\cL,\vc)$ stabilizes the queues $(Q_{nj}^d)_{n,j,d}$, if there exists any loop-free solution $(\vX[t],\vR[t])$ such that $\vx\in\Lambda(\cN,\cL,\vc-\vepsilon)$, $\forall\vepsilon\succ 0$.\label{thm:stability}
\end{proposition}
\begin{proof}
The detailed proof is provided in Appendix~\ref{sec:stability}.
\end{proof}

\textit{Remark:} To stabilize the network, the proposed delay-aware scheduler benefits from the special properties of the token count process. In fact, we have proved that the delay-aware scheduler structure can be extended to utilize a general service process that satisfies certain regulation constraints and still achieves network stability. The existence and characterization of other delay-aware service processes remains an open research problem.

\subsection{Delay-aware cross-layer policy}
The approaches we proposed in Section~\ref{sec:darouting} and~\ref{sec:regulation} can be used in the multi-layer architecture of Section~\ref{sec:architecture}. To attain a utility maximization solution that reduces the expected end-to-end delay, we incorporate back-pressure solution to problem~(\ref{eqn:stochasticBP}) together with delay-aware routing and delay-aware scheduling in the policy design. In the following policy, we call the queue length seen by DTBP the price, which is a value reflecting the queue length price instead of being an actual queue.

\vspace{1ex} \hrule \vspace{1ex} \textbf{Delay-Aware Cross-Layer Policy}
\begin{itemize}
\item The first layer runs DTBP policy (see Section~\ref{sec:background}) to generate $(\tilde{\vX}[t],\tilde{\vR}[t])$ for $(\cN,\cL,\vc-\vepsilon)$ in a virtual implementation using counters for prices.
\item The second layer performs the delay-aware link rate mapping from the first layer and generates $\hat{\vR}[t]$ as in~(\ref{eqn:mapping}) with $\vR[t]=\tilde{\vR}[t]$.
\item The third layer uses delay-aware schedulers with regulated service in the network $(\cN,\cL,\vc)$. Solution $\vX[t]=\tilde{\vX}[t]$ is received from the first layer. Solution $\bar{\vR}[t]=\hat{\vR}[t]$ is received from the second layer. Then the delay-aware scheduling is performed as described in Section~\ref{sec:regulation} which determines the actual queueing dynamics.
\end{itemize}
\vspace{1ex} \hrule \vspace{1ex}
This cross-layer policy is an online
policy that starts to function while the DTBP is converging. All
layers run different algorithms in parallel and dynamic solution
updates can be performed among layers.

The cross-layer policy inherits optimality (Proposition~\ref{thm:optimal}) and loop-freeness (Proposition~\ref{thm:loopfree}) characteristics of the routing component, and the stability (Proposition~\ref{thm:stability}) of the delay-aware scheduling component, which lead to the following fundamental result:
\begin{theorem}
Delay-aware cross-layer policy results in loop-free and stable rate assignments that arbitrarily approaches maximum utility solutions for a network $(\cN,\cL,\vc)$.\label{thm:policy}
\end{theorem}

\section{Numerical Results}\label{sec:numerical}
\begin{figure*}[!t]
\centering
\subfigure[]{\includegraphics[width=2.3in]{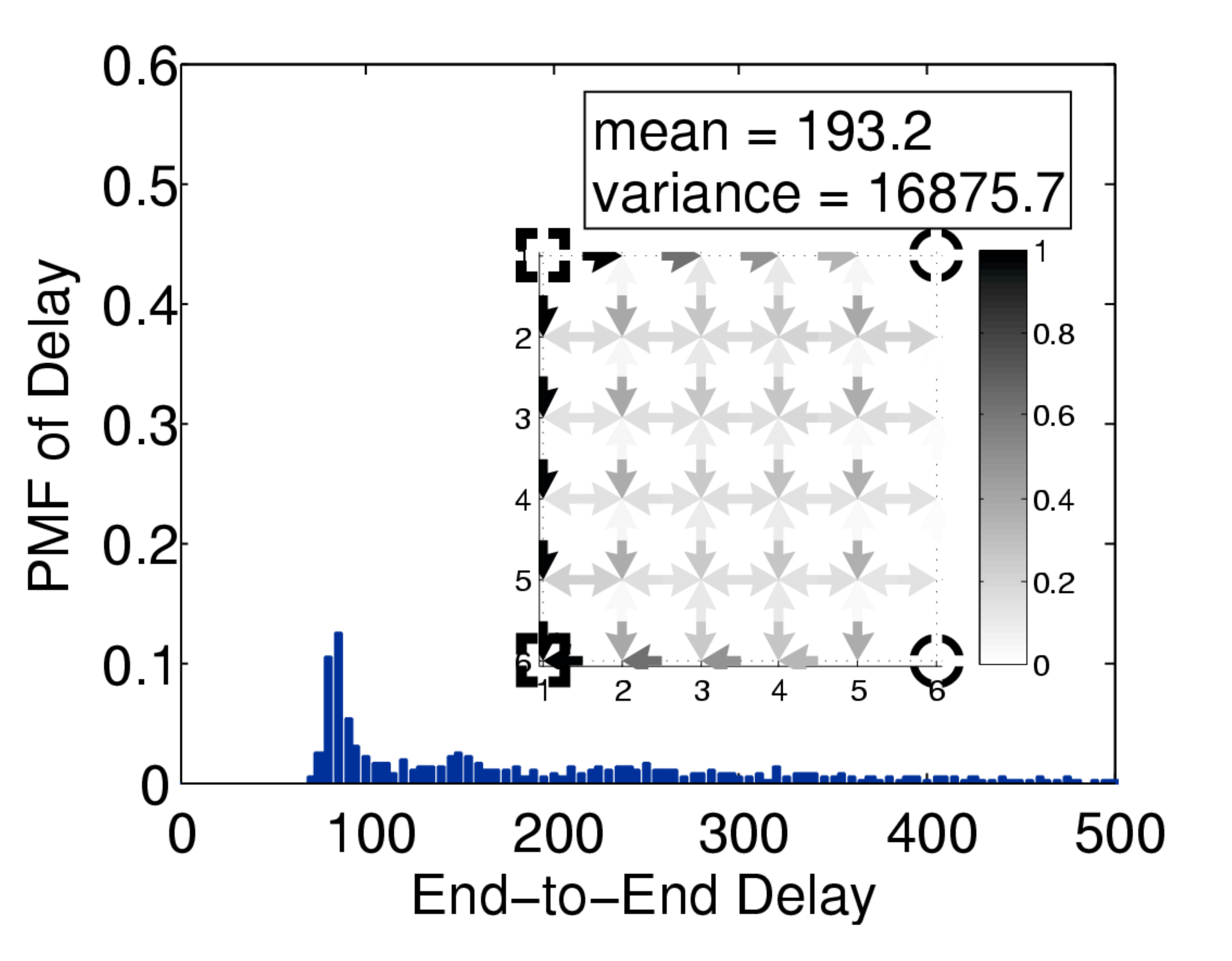}}
\label{fig:s1bpdp}
\hfil
\subfigure[]{\includegraphics[width=2.3in]{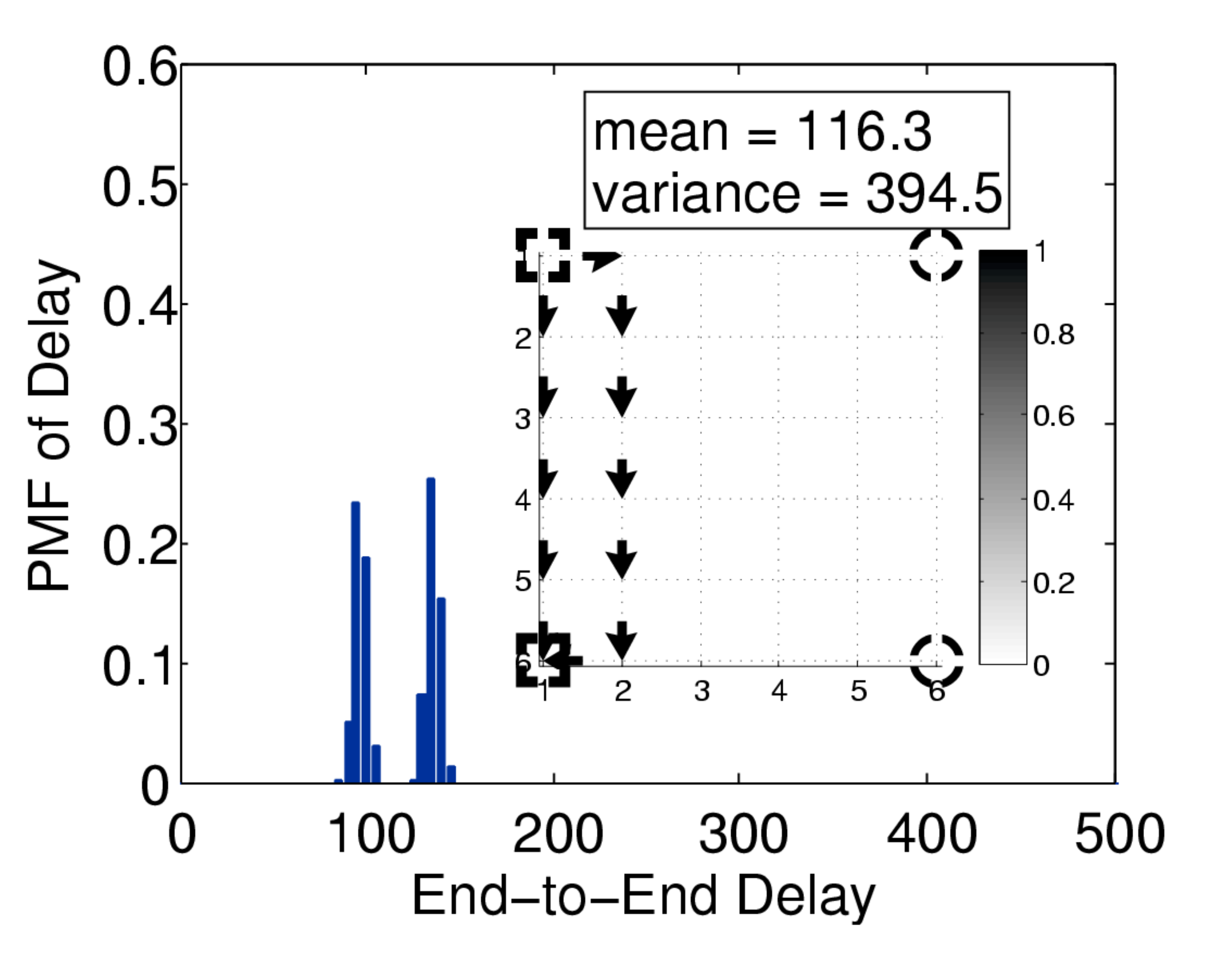}}
\label{fig:s1iqlbdp}
\hfil
\subfigure[]{\includegraphics[width=2.3in]{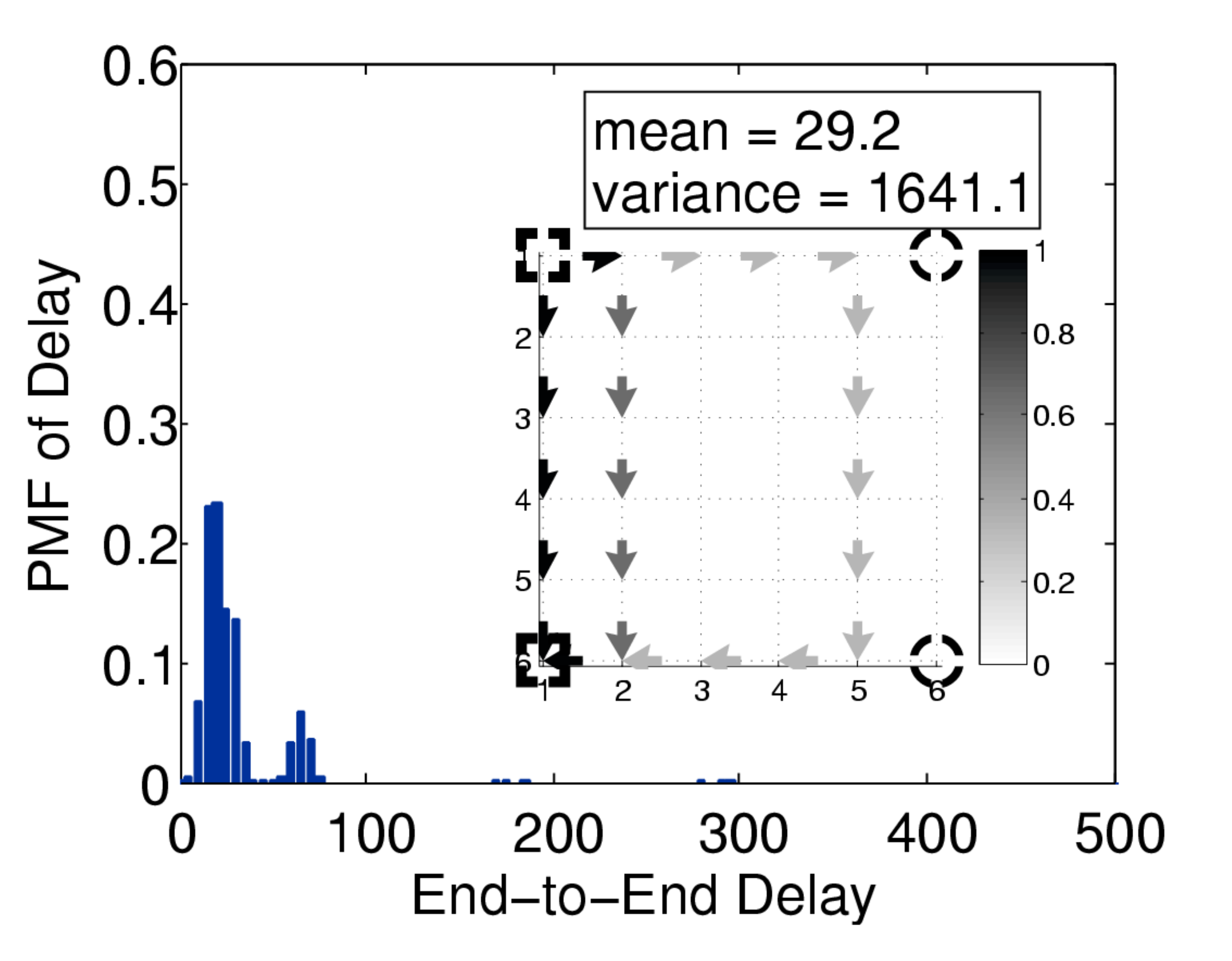}}
\label{fig:s1tbdp}
\caption{Scenario 1 Delay Performance (a) Back-Pressure Policy (b) Min-Resource Algorithm (c) Delay-Aware Cross-Layer Policy}
\label{fig:scenario1}
\end{figure*}

\begin{figure*}[!t]
\centering
\subfigure[]{\includegraphics[width=2.3in]{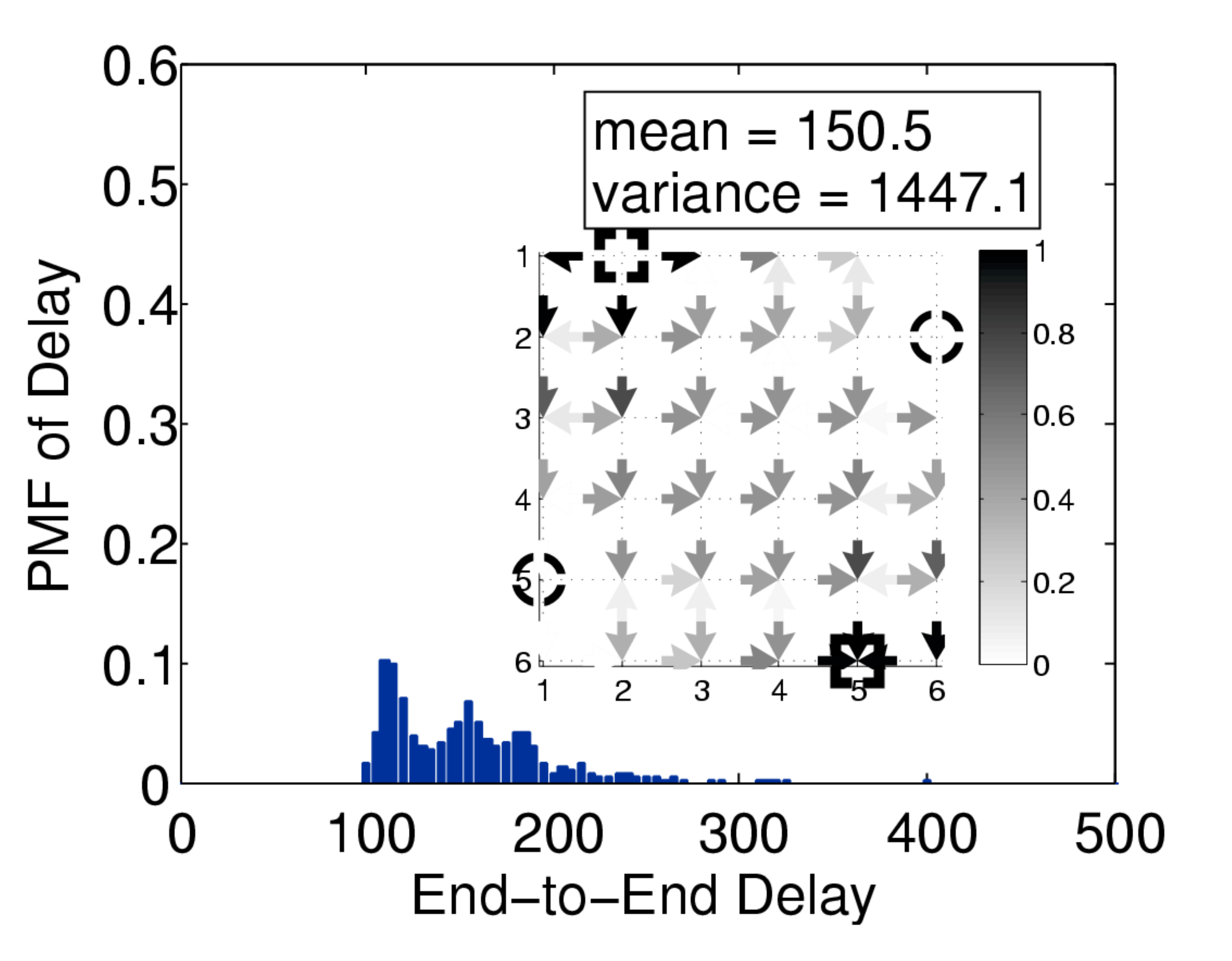}}
\label{fig:s3bpdp}
\hfil
\subfigure[]{\includegraphics[width=2.3in]{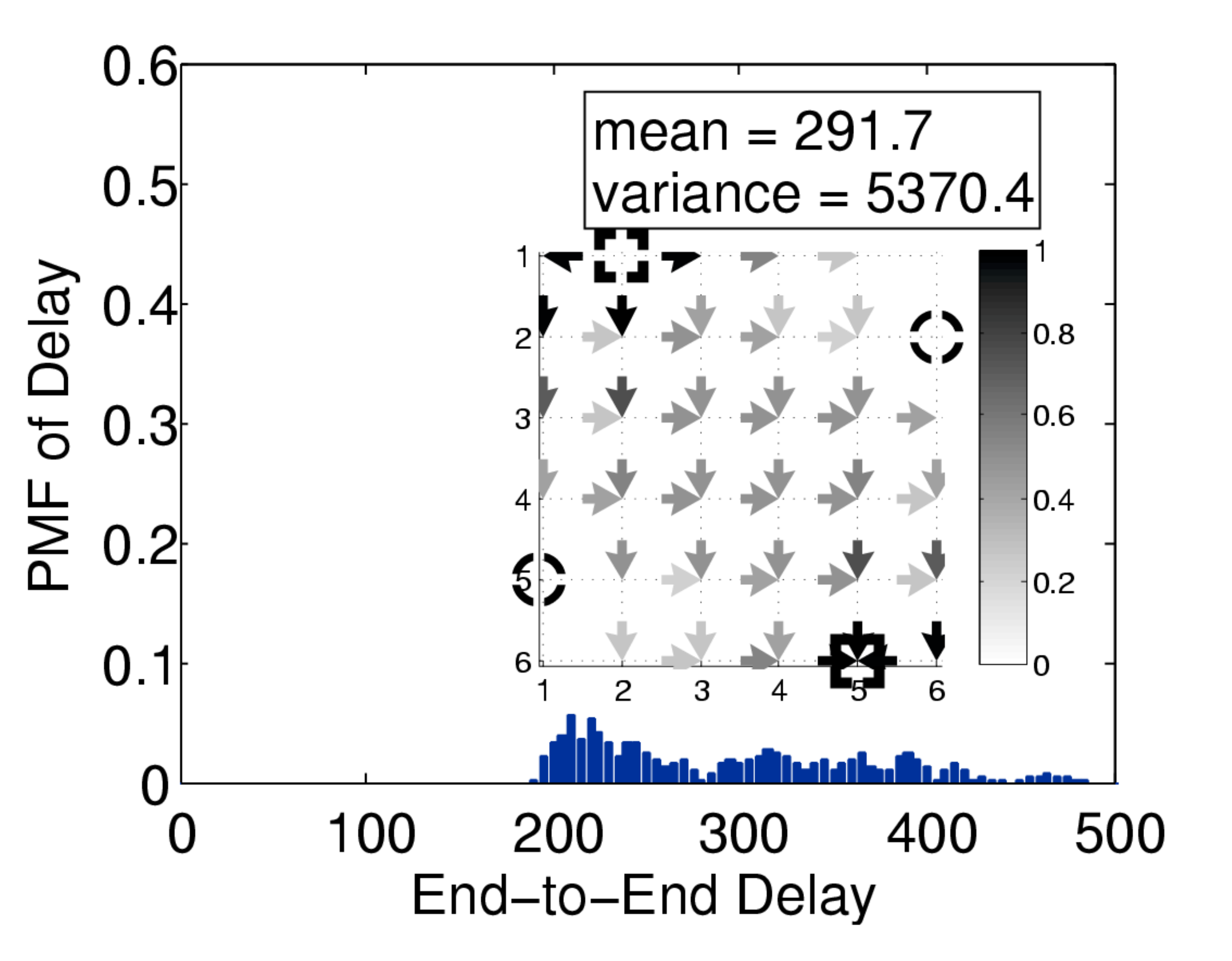}}
\label{fig:s3iqlbdp}
\hfil
\subfigure[]{\includegraphics[width=2.3in]{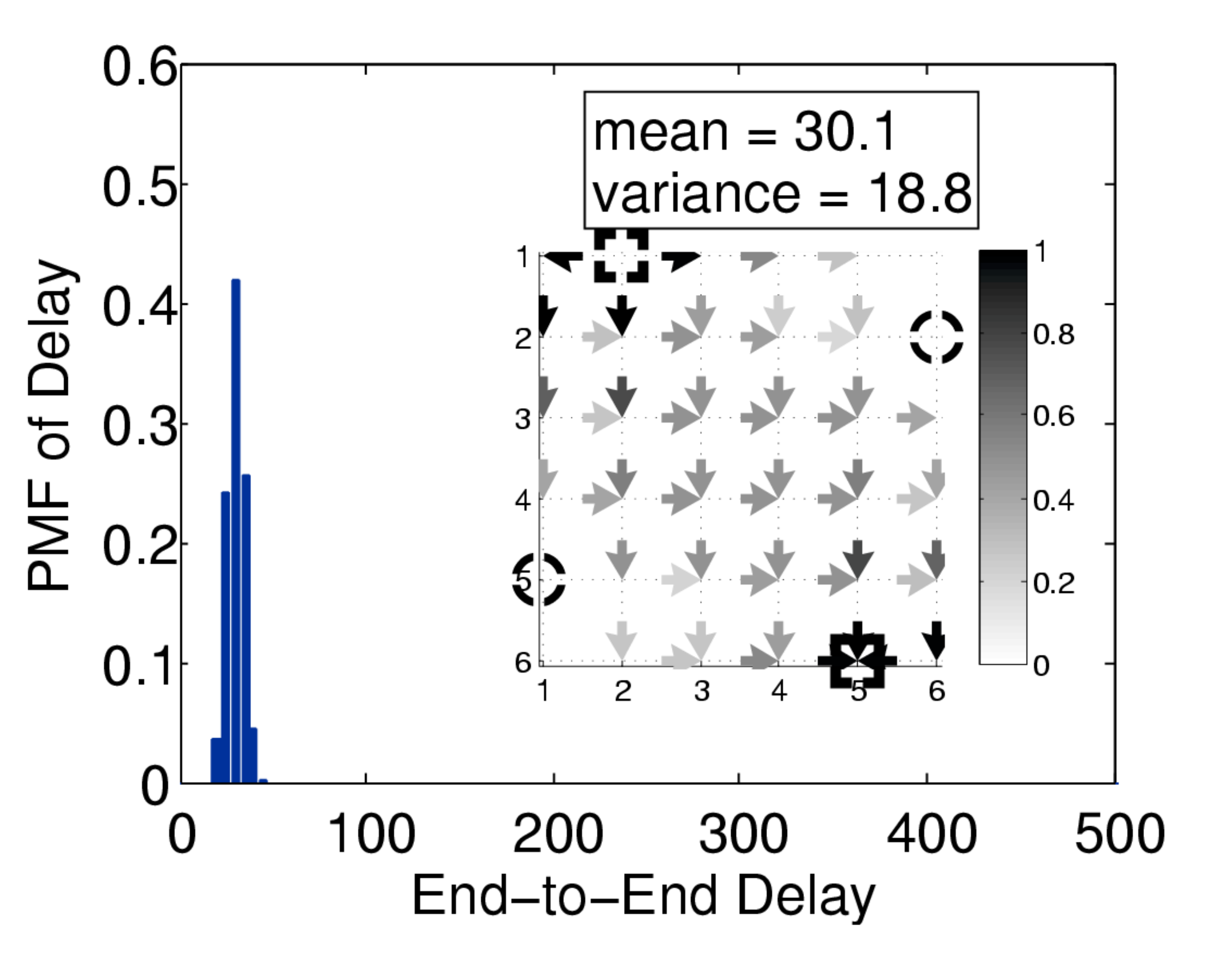}}
\label{fig:s3tbdp}
\caption{Scenario 2 Delay Performance (a) Back-Pressure Policy (b) Min-Resource Algorithm (c) Delay-Aware Cross-Layer Policy}
\label{fig:scenario2}
\end{figure*}

In this section, we present numerical results for our proposed delay-aware cross-layer policy. The cross-layer policy is an online policy starting to function as DTBP is converging. All layers run different algorithms parallelly and perform periodical solution updates in implementation. We compare the delay performance of our policy with the back-pressure policy, which has long term optimality guarantees, and the more recent min-resource algorithm (see Section~\ref{sec:background}), which has superior routing characteristics over back-pressure policy. As a representative topology, simulations are carried out in a $6\times6$ grid network. Simulation duration is 30000 time slots. Moving average method is used to update $\hat{\vR}[t]$ as follows:
\begin{align*}
\hat{R}_{nj}^d[t]={}&\left[\frac{1}{W}\sum_{\tau = t_0-W}^{t_0}\left(R_{nj}^d[\tau]-R_{jn}^d[\tau]\right)\right]^+,\\
&\qquad\qquad\qquad\qquad\qquad\qquad t_0\leq t< t_0+T,
\end{align*}
where $W$ is the window size and $T$ is update period. $W$ and $T$ are both set to 5000. We assume that all links have unit capacity. The $K$ parameter for the flow controller of the back-pressure policy is chosen to be 200. We implement min-resource algorithm as described in Section~\ref{sec:background} with $M=3$. For fair comparison, all queues are implemented as per-commodity queues.

Two representative scenarios are considered to demonstrate the performance difference of the three algorithms. \textit{Scenario 1} has parallel traffics with two commodities. \textit{Scenario 2} contains cross traffics where each source node maintains throughput up to 3 compared to 2 in Scenario 1. The links in the center \(4\times4\) square are fully utilized. Scenario 2 has less excess capacity compared to Scenario 1. 

For better demonstration, the routing information is embedded in delay distribution graph. Since the source-destination pairs are symmetric, results are given for one commodity only. Squares and circles indicate source and destination locations for commodity 1 and 2, respectively, both flowing from top to bottom. In Scenario 1, as is shown in Fig.~\ref{fig:scenario1}(a), the back-pressure policy results in loops. The delay distribution is heavy-tailed which also indicates that packets may experience looping. Fig.~\ref{fig:scenario1}(b) demonstrates that the min-resource algorithm utilizes minimal number of links, thus attaining shortest paths from source to destination for each commodity. Since loops are avoided, the delay distribution is improved with the delay peaks moved towards lower values. Fig.~\ref{fig:scenario1}(c) shows the network topology for delay-aware cross-layer policy. The selected paths are guaranteed to be loop-free. The same figure also shows that the delay distribution is multimodal. From this distribution, one observes that the majority of packets follow the shortest path and the others are routed over longer paths. The mean delay performance is improved over other two solutions.

In Scenario 2, we demonstrate a situation where the network is heavily utilized. The existence of excess capacity is a necessary condition for back-pressure policy to form loops. Within a network with less excess capacity, back-pressure policy has less resources for loops, which leaves less chance for routing improvement. Fig.~\ref{fig:scenario2}(a) shows the back-pressure result for such a network with increased throughput. In Fig.~\ref{fig:scenario2}(b), we observe no improvement for min-resource algorithm since the network is already heavily loaded. Min-resource algorithm cannot further minimize resource usage without sacrificing throughput. In Fig.~\ref{fig:scenario2}(c), we still observe significant delay improvements in the mean and variance of delay. The result suggests that the delay-aware cross-layer policy is effective in reducing delay regardless of network utilization.

We further study the effect of the design parameter $K$ used in the DTBP policy. It is known that the allocated flow rates approach the optimal rate $\vx^*$ as $\textstyle K\rightarrow\infty$. However, this also increases the equilibrium queue lengths in the DTBP and causes delay degradation (\cite{erysri05, neemodli05}, etc.)

\begin{figure}[!th]
\centering
\includegraphics[width=2.5in]{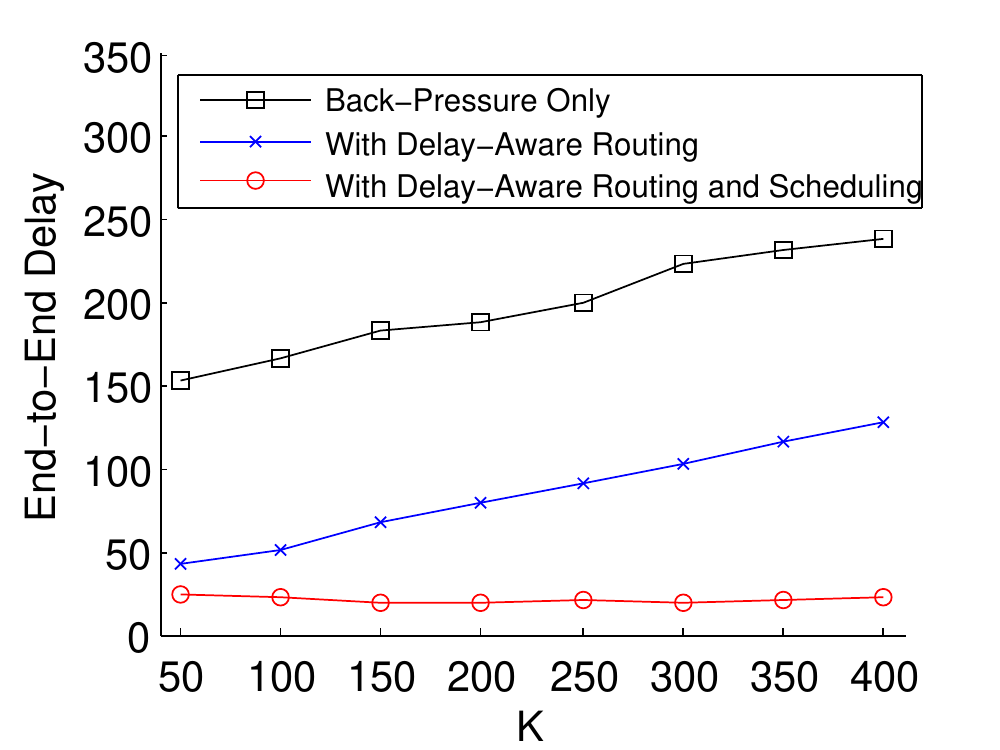}
\caption{End-to-End Delay under Different K Values}
\label{fig:kindependence}
\end{figure}
The delay-aware cross-layer policy offers a solution to this dilemma. In multi-layer architecture, the delay-aware cross-layer policy keeps real layer queue length separated from back-pressure layer queue length. Then the dependence on $K$ is relaxed. Fig.~\ref{fig:kindependence} shows that in Scenario 1, the sensitivity of the delay on $K$ is significantly reduced with delay-aware cross-layer policy. As expected, the delay under back-pressure policy grows as \(K\) increases. In this figure, we also show that the end-to-end delay can be reduced by eliminating loops alone. However, further delay improvement can be achieved by utilizing the delay-aware scheduling with delay-aware routing.

\section{Conclusion}\label{sec:conclusion}

In this work, we exposed the delay deficiencies of many long-term
optimal policies when operated in multi-hop networks, and identified
several unexploited design choices within the routing and scheduling
strategy spaces that yield drastic delay improvements. In
particular, we exploit: the flexibility in link rate assignments to
eliminate loops in routing; and the service shaping opportunities in
scheduling to reduce queue-sizes, while preserving the long-term
optimality characteristics.

Along with these algorithms, we also present a generic delay-aware
design framework that can be used to develop other long-term optimal
algorithms with favorable delay characteristics. This framework
provides a unique and methodical approach to the design of modular
solutions where end-to-end delay characteristics can be improved
without sacrificing from long-term optimality. As such, our
framework offers a promising direction for tackling the challenging
problem of delay-aware algorithm design for wireless multi-hop
networks, as well. To that end, in our future work, we will extend
the insights and techniques developed in this work to the wireless
domain, while accounting for the intrinsic challenges associated
with the interference-limited nature of wireless communications.

%

\bibliographystyle{IEEEtran}
\bibliography{IEEEabrv,./refs}

\appendix
\newcommand{\vA}{\boldsymbol{A}}
\newcommand{\vpi}{\boldsymbol{\pi}}

\subsection{Proof of Proposition \ref{thm:optimal}}\label{sec:optimal}
Let the optimal solution to (\ref{eqn:objectproblem}) be $(\vx^*,\vr^*)$, we define
\begin{eqnarray}
&&\hat{\vx}^*=\vx^*;\label{eqn:hat_vx}\\
&&\hat{\vr}^*\mbox{ as defined in Equations (\ref{eqn:r_ij_hat}) and (\ref{eqn:r_ji_hat})}. \nonumber
\end{eqnarray}
Since $\hat{\vx}^*=\vx^*$, the optimal value of the objective function (\ref{eqn:objectproblem}) is not affected. Thus, to show that $(\hat{\vx}^*,\hat{\vr}^*)$ is also an optimal solution to (\ref{eqn:objectproblem}), it is sufficient to show that the pair $(\hat{\vx}^*,\hat{\vr}^*)$ satisfies the constraints (\ref{eqn:xpositive})--(\ref{eqn:flowconservation}).

$\bullet$ Verifying Equation~(\ref{eqn:xpositive}) and (\ref{eqn:rpositive}):

$\hat{x}_s^{*d}\geq0$ because $x_s^{*d}\geq0$ for all $s,d\in\cD$; $\hat{r}_{ij}^{*d}\geq0$ for all $(i,j)\in\cL$ and $d\in\cD$ by its definition in Equations~(\ref{eqn:r_ij_hat}) and (\ref{eqn:r_ji_hat}).

$\bullet$ Verifying Equation~(\ref{eqn:capacityconstraint}):

By the definition of $\hat{r}_{ij}^{*d}$ and the fact that $r_{ij}^{*d}\geq 0$, $\forall (i,j)\in\cL$, $\forall d\in\cD$, we have \(\hat{r}_{ij}^{*d} \leq r_{ij}^{*d},\, \forall (i,j)\in\cL,\, \forall d\in\cD\). Then \(\sum_d\hat{r}_{ij}^{*d} \leq \sum_dr_{ij}^{*d}\). So \(\sum_d\hat{r}_{ij}^{*d} \leq c_{ij}\).

$\bullet$ Verifying Equation~(\ref{eqn:flowconservation}):

Noticing that for any real number $Z$, \((Z)^+=\frac{1}{2}(Z+|Z|)\), Equations (\ref{eqn:r_ij_hat}) and (\ref{eqn:r_ji_hat}) can be rewritten as:
\begin{eqnarray*}
\hat{r}_{ij}^{*d} &=& \frac{1}{2} \left((r_{ij}^{*d}-r_{ji}^{*d})+ |r_{ij}^{*d}-r_{ji}^{*d}|\right);\\
\hat{r}_{ji}^{*d} &=& \frac{1}{2} \left((r_{ji}^{*d}-r_{ij}^{*d})+ |r_{ji}^{*d}-r_{ij}^{*d}|\right).
\end{eqnarray*}
Subtracting the above two equations, we have
\begin{eqnarray}
\hat{r}_{ij}^{*d} - \hat{r}_{ji}^{*d} = r_{ij}^{*d}-r_{ji}^{*d}\label{eqn:relation}.
\end{eqnarray}
Thus, the following equation holds:
\begin{eqnarray*}
&& \hat{x}_n^{*d} + \sum_{\langle m,i\rangle\in\cL} \hat{r}_{mi}^{*d} -\sum_{\itoj\in\cL}\hat{r}_{ij}^{*d}\\
&\stackrel{(a)}{=}& \hat{x}_i^{*d} + \sum_{j}\left( \hat{r}_{ij}^{*d} - \hat{r}_{ji}^{*d} \right)\\
&\stackrel{(b)}{=}& x_i^{*d} + \sum_{j}\left( r_{ij}^{*d} - r_{ji}^{*d} \right)
\leq 0,
\end{eqnarray*}
where (a) follows from (\ref{eqn:relation}) and (b) follows from (\ref{eqn:hat_vx}).

Therefore, $(\hat{\vx}^*,\hat{\vr}^*)$ is also an optimal solution to (\ref{eqn:objectproblem}).

\subsection{Proof of Lemma \ref{lmm:3hops}}\label{sec:3hops}

Lemma \ref{lmm:3hops} can be proven by constructing the queue-evolution Markov chain under the DTBP algorithm and studying the steady-state behavior of the queue-length evolution from zero state.

To simplify the notation, we denote the nodes as node $0$ to node $n$ in a $n$-hop tandem network, where node $0$ is the destination, and node $n$ is the source. Since only one flow is assumed to exist, we omit the superscript $d$ indicating the destination of the flow.

\textbf{Case 1: }$n=1$. In this case, node $1$ receives a packet with probability $a$ in each time-slot. Whenever node $1$ has one packet, it sends the packet to the destination. Thus $Prob\{P_1=1\}=a$, and $Prob\{P_1=0\}=1-a$, thus $$\bar{P}_1^*=a>0=\bar{P}_0^*$$ when $a\in(0,1)$.

\textbf{Case 2: }$n=2$. For brevity, we denote the state of the Markov chain $(P_2,P_1)$ as in the Table \ref{tab:2hop}.
\begin{table}[!ht]
\begin{center}
\caption{The states in the Markov chain of a 2-hop tandem network}\label{tab:2hop}
\begin{tabular}{c|c||c|c}\hline\hline
$(P_2,P_1)$& State & $(P_2,P_1)$& State\\\hline\hline
$(0,0)$ & 1 & $(1,1)$ & 4 \\\hline
$(0,1)$ & 2 & $(2,0)$ & 5\\\hline
$(1,0)$ & 3 & $(2,1)$ & 6\\\hline\hline
\end{tabular}
\end{center}
\end{table}
Since the Markov chain itself is messy to plot, we only present its state transition matrix
\begin{eqnarray*}
\vA = \left[
\begin{array}{cccccc}
b & a & \cdot & \cdot & \cdot & \cdot\\
\cdot & \cdot & b & \cdot & a & \cdot\\
\cdot & b & \cdot & a & \cdot & \cdot\\
\cdot & \cdot & b & \cdot & a & \cdot\\
\cdot & \cdot & \cdot & b & \cdot & a\\
\cdot & \cdot & \cdot & b & \cdot & a
\end{array}
\right],
\end{eqnarray*}
where $b=1-a$, and we use ``$\cdot$'' to represent 0 for clarity. Solve the equation $$\vA\vpi=\vpi$$ and the probability constraint $$\sum_i \pi_i=1$$ to get the steady-state distribution of this Markov chain, and the average queue-length can be calculated by
\begin{eqnarray*}
\bar{P}_1^* &=& \pi_2 + \pi_4 + \pi_6,\\
\bar{P}_2^* &=& \pi_3+\pi_4 + 2(\pi_5+\pi_6).
\end{eqnarray*}
It can be verified that $$\bar{P}_2^*>\bar{P}_1^*>\bar{P}_0^*=0$$ when $a\in(0,1)$. Shown in Fig. \ref{fig:3hopsteady} is the average queue-lengths as a function of $a$ in a 2-hop tandem network.

\begin{figure}[!h]
\begin{center}
\includegraphics[width=2.5in]{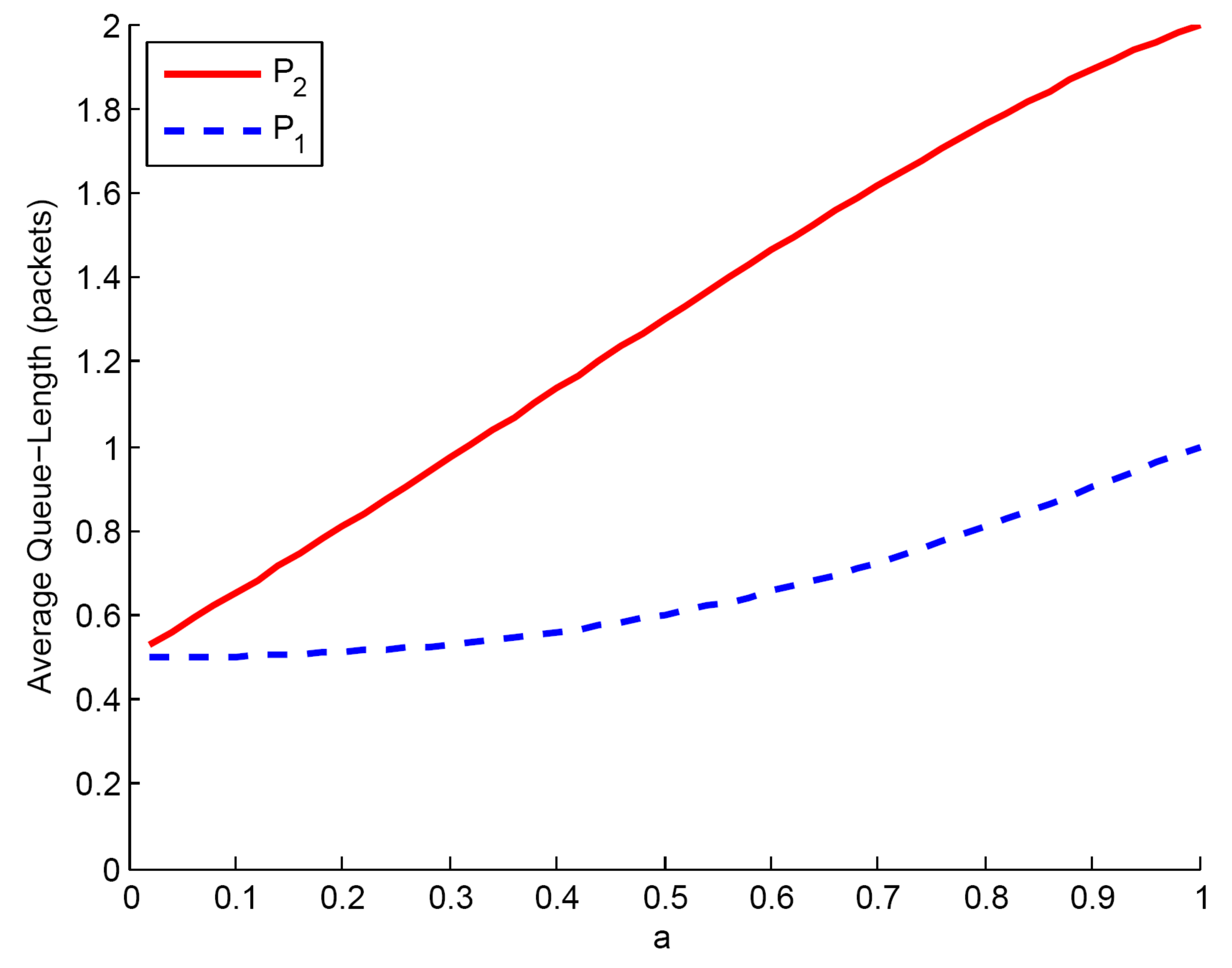}
\caption{The average queue-length in a 2-hop tandem network under Bernoulli($a$) arrival. The queue-lengths are strictly decreasing from source to destination.} \label{fig:2hopsteady}
\end{center}
\end{figure}

\textbf{Case 3: }$n=3$. For brevity, we denote the state of the Markov chain $(P_3,P_2,P_1)$ as in the Table \ref{tab:3hop}.
\begin{table}[!ht]
\begin{center}
\caption{The states in the Markov chain of a 3-hop tandem network}\label{tab:3hop}
\begin{tabular}{c|c||c|c}\hline\hline
$(P_3,P_2,P_1)$& State & $(P_3,P_2,P_1)$& State\\\hline\hline
$(0,0,0)$ & 1 & $(1,0,1)$ & 9\\\hline
$(1,0,0)$ & 2 & $(2,0,1)$ & 10\\\hline
$(0,1,0)$ & 3 & $(3,0,1)$ & 11\\\hline
$(1,1,0)$ & 4 & $(2,1,1)$ & 12\\\hline
$(0,2,0)$ & 5 & $(3,1,1)$ & 13\\\hline
$(1,2,0)$ & 6 & $(2,2,1)$ & 14\\\hline
$(2,2,0)$ & 7 & $(3,2,1)$ & 15\\\hline
$(3,2,0)$ & 8 & &\\\hline\hline
\end{tabular}
\end{center}
\end{table}
With this numbering of the states, the state transition matrix is given by
\begin{eqnarray*}
&\vA =\\
&\left[
\begin{array}{ccccccccccccccc}
b & a & \cdot & \cdot & \cdot & \cdot & \cdot & \cdot & \cdot & \cdot & \cdot & \cdot & \cdot & \cdot & \cdot\\
\cdot & \cdot & b & a & \cdot & \cdot & \cdot & \cdot & \cdot & \cdot & \cdot & \cdot & \cdot & \cdot & \cdot\\
\cdot & \cdot & \cdot & \cdot & \cdot & \cdot & \cdot & \cdot & b & a & \cdot & \cdot & \cdot & \cdot & \cdot\\
\cdot & \cdot & \cdot & \cdot & \cdot & \cdot & \cdot & \cdot & b & a & \cdot & \cdot & \cdot & \cdot & \cdot\\
\cdot & \cdot & \cdot & \cdot & \cdot & \cdot & \cdot & \cdot & b & a & \cdot & \cdot & \cdot & \cdot & \cdot\\
\cdot & \cdot & \cdot & \cdot & \cdot & \cdot & \cdot & \cdot & \cdot & b & a & \cdot & \cdot & \cdot & \cdot\\
\cdot & \cdot & \cdot & \cdot & \cdot & \cdot & \cdot & \cdot & \cdot & \cdot & \cdot & b & a & \cdot & \cdot\\
\cdot & \cdot & \cdot & \cdot & \cdot & \cdot & \cdot & \cdot & \cdot & \cdot & \cdot & \cdot & \cdot & b & a\\
\cdot & \cdot & \cdot & \cdot & b & a & \cdot & \cdot & \cdot & \cdot & \cdot & \cdot & \cdot & \cdot & \cdot\\
\cdot & \cdot & \cdot & \cdot & \cdot & b & a & \cdot & \cdot & \cdot & \cdot & \cdot & \cdot & \cdot & \cdot\\
\cdot & \cdot & \cdot & \cdot & \cdot & \cdot & b & a & \cdot & \cdot & \cdot & \cdot & \cdot & \cdot & \cdot\\
\cdot & \cdot & \cdot & \cdot & \cdot & b & a & \cdot & \cdot & \cdot & \cdot & \cdot & \cdot & \cdot & \cdot\\
\cdot & \cdot & \cdot & \cdot & \cdot & \cdot & b & a & \cdot & \cdot & \cdot & \cdot & \cdot & \cdot & \cdot\\
\cdot & \cdot & \cdot & \cdot & \cdot & \cdot & \cdot & \cdot & \cdot & \cdot & \cdot & b & a & \cdot & \cdot\\
\cdot & \cdot & \cdot & \cdot & \cdot & \cdot & \cdot & \cdot & \cdot & \cdot & \cdot & \cdot & \cdot & b & a
\end{array}
\right],
\end{eqnarray*}
where $b=1-a$, and we use ``$\cdot$'' to represent 0 for clarity. Solve the equation $$\vA\vpi=\vpi$$ and the probability constraint $$\sum_i \pi_i=1$$ to get the steady-state distribution of this Markov chain, and the average queue-length can be calculated by
\begin{eqnarray*}
\bar{P}_1^* &=& \sum_{i=9}^{15}\pi_i,\\
\bar{P}_2^* &=& \pi_3+\pi_4 + \pi_{12}+\pi_{13}\\
&& + 2(\pi_5+\pi_6+\pi_7+\pi_8+\pi_{14}+\pi_{15}),\\
\bar{P}_3^* &=& \pi_2+\pi_4 + \pi_6+\pi_9\\
&& + 2(\pi_7+\pi_{10} + \pi_{12}+\pi_{14})\\
&& + 3(\pi_8+\pi_{11} + \pi_{13}+\pi_{15}).
\end{eqnarray*}
It can be verified analytically or numerically that $$\bar{P}_3^*>\bar{P}_2^*>\bar{P}_1^*>\bar{P}_0^*=0$$ when $a\in(0,1)$. Fig. \ref{fig:3hopsteady} shows the average queue-lengths as a function of $a$ in a 3-hop tandem network.

\begin{figure}[!h]
\begin{center}
\includegraphics[width=2.5in]{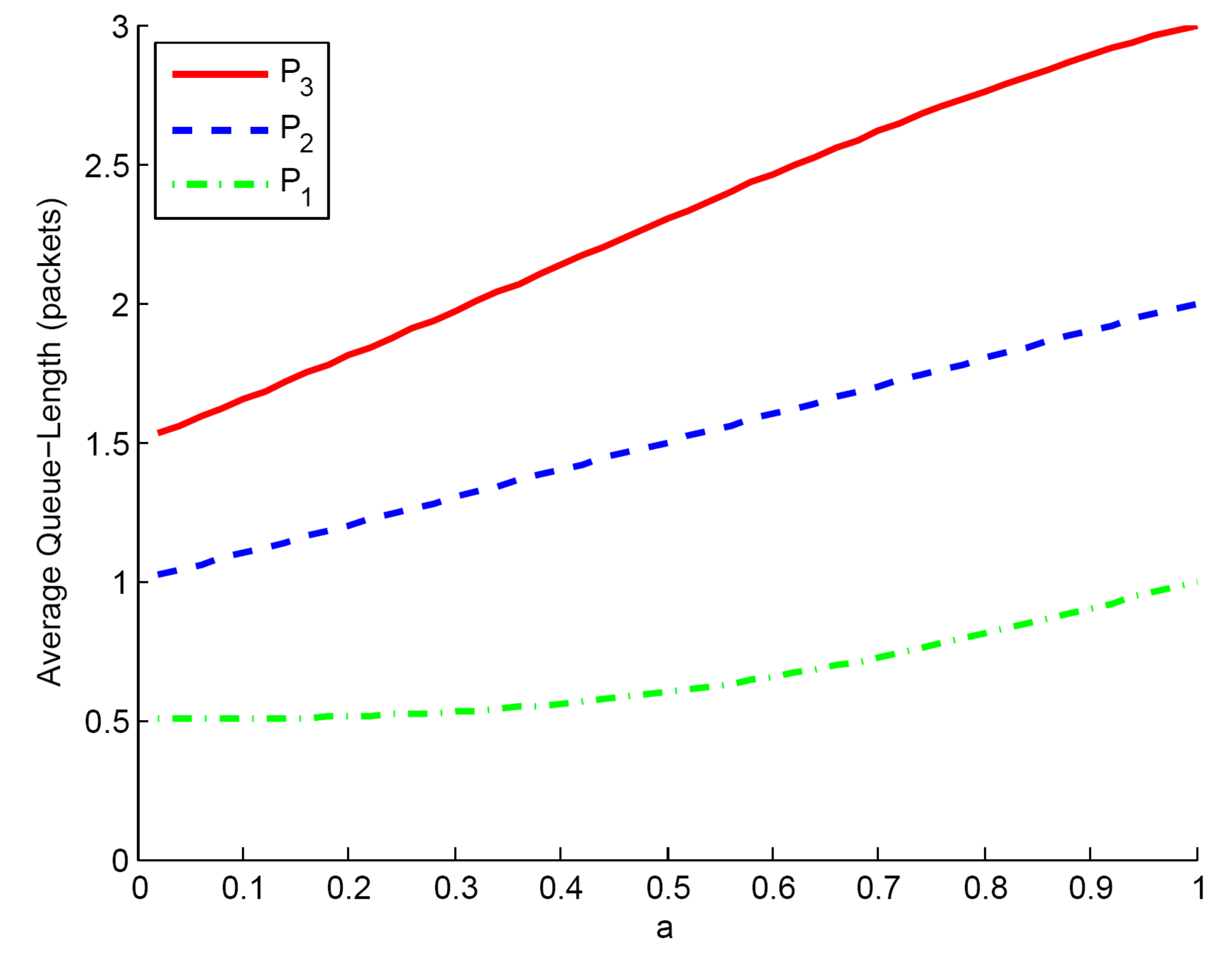}
\caption{The average queue-length in a 3-hop tandem network under Bernoulli($a$) arrival. The queue-lengths are strictly decreasing from source to destination.} \label{fig:3hopsteady}
\end{center}
\end{figure}

\emph{Remark: }As we can see from the above analysis, the dimension and the number of states in the Markov chain grow fast as $n$ increases. It is cumbersome and impractical to precisely analysis the Markov chain to prove the generalization of Lemma \ref{lmm:3hops}.

\subsection{Proof of Lemma \ref{lmm:linearNetwork}}\label{sec:linearNetwork}

To simplify the notation, we denote the nodes as node $0$ to node $n$ in a $n$-hop tandem network, where node $0$ is the destination, and node $n$ is the source. Since only one flow is assumed to exist, we omit the superscript $d$ indicating the destination of the flow.

Let $\epsilon$ be defined as $\epsilon=1-a$. By the feasibility constraint (\ref{eqn:capacityconstraint}) and the flow conservation constraint (\ref{eqn:flowconservation}), we have
\begin{eqnarray*}
r_{i,i-1}^*+r_{i-1,i}^* &\leq& 1\\
r_{i,i-1}^*-r_{i-1,i}^* &\geq& 1-\epsilon
\end{eqnarray*}
for all $i=1,2,\cdots,n$. Thus we have $r_{i-1,i}\leq\epsilon/2$ for all $i$.

Also, note that under DTBP algorithm, $|P_i[t]-P_{i-1}[t]|\leq3$ for all $i=1,2,\cdots,n-1$ and any $t$. This is because each node $i$ can get at most $2$ packets from its neighbors (2 is the degree of all intermediate nodes in a tandem network), while its neighboring nodes can send out at most $1$ more packet to other nodes. To receive one packet from both of its neighbor, the queue length at node $i$ is at least one less than that of its neighbor. This implies that the steady state distribution satisfies
\begin{eqnarray*}
\pi(\vP)\indicator_{\{|P_i-P_{i-1}|>3\}}(\vP)=0,\,\forall i = 1,2,\cdots,n-1,
\end{eqnarray*}
where $\indicator_{\{\mathcal{A}\}}(\vP)$ is the indicator function of a set $\mathcal{A}$ and $\indicator_{\{\mathcal{A}\}}(\vP)=1$ if $\vP\in\mathcal{A}$ and $0$ otherwise. Note that $P_{i-1}[t]-P_i[t]\leq3$ also holds for the source node $n$.

Thus, we have
\begin{eqnarray*}
&&\bar{P}_i^*-\bar{P}_{i-1}^*\\
&=& \sum_{\vP}\pi(\vP)(P_i-P_{i-1})\\
&=& \sum_{\vP}\pi(\vP)\indicator_{\{P_i>P_{i-1}\}}(\vP)(P_i-P_{i-1})\\
&& - \sum_{\vP}\pi(\vP)\indicator_{\{P_i<P_{i-1}\}}(\vP)(P_{i-1}-P_i)\\
&\geq_{(a)}& \sum_{\vP}\pi(\vP)\indicator_{\{P_i>P_{i-1}\}}(\vP)\\
&& - 3\sum_{\vP}\pi(\vP)\indicator_{\{P_i<P_{i-1}\}}(\vP)\\
&=_{(b)}& r_{i,i-1}^*-3r_{i-1,i}^*\\
&\geq & 1-\epsilon-2r_{i-1,i}^*\\
&\geq & 1-2\epsilon\\
&>& 0 \mbox{ if }\epsilon<\frac{1}{2},
\end{eqnarray*}
where (a) holds because $P_i-P_{i-1}\geq1$ on the set $\{P_i>P_{i-1}\}$ and $P_{i-1}-P_i\leq3$ from the discussion above, and (b) holds from the definition of the DTBP algorithm. This argument holds for all $i=1,2,\cdots,n$, thus the lemma is proven.

\subsection{Proof of Proposition \ref{thm:stability}}\label{sec:stability}
\subsubsection{Proof of Lemma~\ref{thm:attinv}}\label{sec:attinv}
First we show that $\cM$ is attractive and that the process $\vM[t]$ evolves towards $\cM$. Let $t$ be such that $\vM[t]\notin\cM$. Then $\sum_{i=1}^nM_i[t]\geq (n+1)c^{th}$. Then $\frac{\sum_{i=1}^dM_i[t]}{n}\geq c^{th}+\frac{c^{th}}{n}>c^{th}$. So there exists $j$ such that $M_j[t]>c^{th}$, which implies $\sum_{i=1}^nD_i[t]=c^{th}$. Then we have
\begin{align*}
&\sum_{i=1}^nM_i[t+1]-\sum_{i=1}^nM_i[t]\\
=&\sum_{i=1}^n\nu_i-\sum_{i=1}^nD_i[t]=\sum_{i=1}^n\nu_i-c^{th}<0.
\end{align*}
Let $T=\lceil\frac{\sum_{i=1}^nM_i[t]-(n+1)c^{th}}{c^{th}-\sum_{i=1}^n\nu_i}\rceil$. Then $\vM[t+T]\in\cM$, which proves $\forall t>0$, $\forall \vM[t]\notin\cM$, there exists $T<\infty$ such that $\vM[t+T]\in\cM$.

Next, we prove that $\cM$ is invariant in $\vM$. Define $\cM'=\bigl\{\vM\in\bR^n:\,nc^{th}\leq\sum_{i=1}^n M_i<(n+1)c^{th}\bigr\}$. We first prove that any $\vM[t]\in\cM\backslash\cM'$, $\vM[t+1]\in\cM$, i.e., $\vM[t]$ cannot escape $\cM\backslash\cM'$ without being in $\cM'$. Then we prove any $\vM[t]\in\cM'$, $\vM[t+1]\in\cM$, which shows that any solution in $\cM'$ cannot leave $\cM$.

To see the first claim holds, for any $\vM[t]\in\cM\backslash\cM'$, $\sum_{i=1}^nM_i[t]<nc^{th}$. Then $\sum_{i=1}^nM_i[t+1]\leq\sum_{i=1}^nM_i[t]+\sum_{i=1}^n\nu_i<(n+1)c^{th}$. So $\vM[t+1]\in\cM$.

Now we prove  any $\vM[t]\in\cM'$, $\vM[t+1]\in\cM$. Let $\vM[t]\in\cM'$. Then $\sum_{i=1}^nM_i[t]\geq nc^{th}$. So there exists $j$ such that $M_j[t]>c^{th}$. Thus $\sum_{i=1}^nD_i[t]=c^{th}$ follows. Together with $\sum_{i=1}^nM_i[t]<(n+1)c^{th}$, we have

\begin{equation*}
\sum_{i=1}^nM_i[t+1]=\sum_{i=1}^nM_i[t]+\sum_{i=1}^n\nu_i-c^{th}<(n+1)c^{th}.
\end{equation*}
So $\vM[t+1]\in\cM$. Hence $\cM$ is invariant in $\vM$. And the lemma follows.


\subsubsection{Proof of Proposition \ref{thm:stability}}\label{sec:stability}
As described in section~\ref{sec:regulation}, we assume that packets are sent from sources to destinations along known multi-paths. We assume that the exogenous arrival rates and the steady state link rates are available from $({\vX}[t], {\vR}[t])$, the loop-free solution such that $\vx\in\Lambda(\cN,\cL,\vc-\vepsilon)$. As $X_n^d[t]$ is the exogenous arrival of commodity $d$ at source node $n$, we assume $\bE[(X_n^d[t])^2]<\infty$.

We acquire the arrival splitting ratio from the back-pressure algorithm in steady state, i.e., the definition in~(\ref{eqn:arrival})(\ref{eqn:regulator}) becomes
\begin{eqnarray}
\mathbb{E}\left[A_{nj}^d[t]\right]&=&\mathbb{E}\left[A_n^d[t]\right]\frac{\bar{R}_{nj}^{d}}{\sum_{m\in\cN}\bar{R}_{nm}^{d}},\label{eqn:arrivalfixed}\\
S_{nj}^d&=&\mathbb{E}\left[A_n^d[t]\right]\frac{\bar{R}_{nj}^{d}}{\sum_{m\in\cN}\bar{R}_{nm}^{d}}+\delta.\label{eqn:regulatorfixed}
\end{eqnarray}
Let $H_{max}<\infty$ be the maximum path length since the paths are loop-free. $\sum_d{S}_{nj}^d$ can be made strictly less than $c$ by choosing sufficiently small $\delta$, i.e., $\delta<\frac{\epsilon}{|H_{max}||\cD|}$ guarantees that the link capacity constraint is satisfied. We assume $\vM[0]\in\bigl\{\vM:\;\sum_{d\in\cD}M_{nj}^d<(\vert\cD\vert+1)c,\;\forall(n,j)\in\cL\bigr\}$.



We define the Lyapunov function as
\begin{equation*}
V(\vQ[t])=\sum_d\sum_{n}\sum_j\bigl(Q_{nj}^d[t]\bigr)^2,
\end{equation*}
where $Q_{nj}^d[t]$ is the queue for commodity $j$ transmissions from node $n$ to node $j$ at time $t$.


For any $T>\frac{\bigl(|\cD|+1\bigr)c}{\delta}$, we consider the $T$ step drift of the Lyapunov function candidate,
\begin{eqnarray*}
\Delta V(\vQ[t])&=&\bE\biggl[V(\vQ[t+T])-V(\vQ[t])\Bigm\vert \vQ[t],\vM[t]\biggr].
\end{eqnarray*}
The following analysis before~(\ref{eqn:firsttermbound}) is similar to the analysis in~\cite{Neely03}. The queueing dynamics over time $T$ satisfies
\begin{equation}
Q_{nj}^d[t+T]\leq\Bigl(Q_{nj}^d[t]-\sum_{\tau=t}^{t+T-1}D_{nj}^d[\tau]\Bigr)^++\sum_{\tau=t}^{t+T-1}A_{nj}^d[\tau],\label{eqn:tstepqueueingdynamics}
\end{equation}
where $A_n^d[t]=X_n^d[t]\mathbbm{1}_{\{n\in\cN_{source}^d\}}+\sum_{k\in\cN}D_{kn}^d[t]$. Squaring both sides of~(\ref{eqn:tstepqueueingdynamics}) and moving terms to the left side, we have
\begin{align}
&\bigl(Q_{nj}^d[t+T]\bigr)^2-\bigl(Q_{nj}^d[t]\bigr)^2\nonumber\\
\leq&\bigl(\sum_{\tau=t}^{t+T-1}D_{nj}^d[\tau]\bigr)^2+\bigl(\sum_{\tau=t}^{t+T-1}A_{nj}^d[\tau]\bigr)^2-2Q_{nj}^d[t]\sum_{\tau=t}^{t+T-1}D_{nj}^d[\tau]\nonumber\\
&+2\sum_{\tau=t}^{t+T-1}A_{nj}^d[\tau]\Bigl(Q_{nj}^d[t]-\sum_{\tau=t}^{t+T-1}D_{nj}^d[\tau]\Bigr)^+\nonumber\\
\leq&\bigl(\sum_{\tau=t}^{t+T-1}D_{nj}^d[\tau]\bigr)^2+\bigl(\sum_{\tau=t}^{t+T-1}A_{nj}^d[\tau]\bigr)^2-2Q_{nj}^d[t]\sum_{\tau=t}^{t+T-1}D_{nj}^d[\tau]\nonumber\\
&+2\sum_{\tau=t}^{t+T-1}A_{nj}^d[\tau]Q_{nj}^d[t]\nonumber\\
=&\bigl(\sum_{\tau=t}^{t+T-1}D_{nj}^d[\tau]\bigr)^2+\bigl(\sum_{\tau=t}^{t+T-1}A_{nj}^d[\tau]\bigr)^2\nonumber\\
&+2Q_{nj}^d[t]\Bigl(\sum_{\tau=t}^{t+T-1}A_{nj}^d[\tau]-\sum_{\tau=t}^{t+T-1}D_{nj}^d[\tau]\Bigr).\nonumber
\end{align}
Then the T step drift of the Lyapunov function satisfies
\begin{align}
&\Delta V(\vQ[t])\nonumber\\
=&\bE\biggl[V(\vQ[t+T])-V(\vQ[t])\Bigm\vert \vQ[t],\vM[t]\biggr]\nonumber\\
=&\bE\biggl[\sum_d\sum_n\sum_j\bigl(Q_{nj}^d[t+T]\bigr)^2\nonumber\\
&-\sum_d\sum_n\sum_j\bigl(Q_{nj}^d[t]\bigr)^2\Bigm\vert \vQ[t],\vM[t]\biggr]\nonumber\\
\leq&\sum_d\sum_n\sum_j\bE\biggl[\bigl(\sum_{\tau=t}^{t+T-1}D_{nj}^d[\tau]\bigr)^2+\bigl(\sum_{\tau=t}^{t+T-1}A_{nj}^d[\tau]\bigr)^2\Bigm\vert \vM[t]\biggr]\nonumber\\
&+2\bE\biggl[\sum_d\sum_n\sum_jQ_{nj}^d[t]\nonumber\\
&\cdot\Bigl(\sum_{\tau=t}^{t+T-1}A_{nj}^d[\tau]-\sum_{\tau=t}^{t+T-1}D_{nj}^d[\tau]\Bigr) \Bigm\vert \vM[t]\biggr]\nonumber\\
\leq&B+2\sum_d\sum_n\sum_jQ_{nj}^d[t]\nonumber\\
&\cdot\bE\biggl[\sum_{\tau=t}^{t+T-1}A_{nj}^d[\tau]-\sum_{\tau=t}^{t+T-1}D_{nj}^d[\tau]\Bigm\vert\vM[t]\biggr].\label{eqn:firsttermbound}
\end{align}
From~(\ref{eqn:arrivalfixed}), we have
\begin{align}
&\bE\biggl[\sum_{\tau=t}^{t+T-1}A_{nj}^d[\tau]\Bigm\vert\vM[t]\biggr]\nonumber\\
=&\sum_{\tau=t}^{t+T-1}\bE\biggl[A_n^d[\tau]\frac{\bar{R}_{nj}^{d}}{\sum_{m\in\cN}\bar{R}_{nm}^{d}}\Bigm\vert\vM[t]\biggr]\nonumber\\
=&\sum_{\tau=t}^{t+T-1}\bE\biggl[X_n^d[\tau]\biggr]\frac{\bar{R}_{nj}^{d}}{\sum_{m\in\cN}\bar{R}_{nm}^{d}}\mathbbm{1}_{\{n\in\cN_{source}^d\}}\nonumber\\
&+\sum_{\tau=t}^{t+T-1}\sum_{k\in\cN}\bE\biggl[D_{kn}^d[\tau]\Bigm\vert\vM[t]\biggr]\frac{\bar{R}_{nj}^{d}}{\sum_{m\in\cN}\bar{R}_{nm}^{d}}.\label{eqn:arrivaldetailed}
\end{align}
We know that $\sum_dS_{nj}^d<1$ and $\sum_{d\in\cD}M_{nj}^d[0]<(\vert\cD\vert+1)c$. By Lemma~\ref{thm:attinv}, we have $\sum_{d\in\cD}M_{nj}^d[t]<(\vert\cD\vert+1)c$, $\forall t>0$. For each $D_{nj}^d[t]$, we have
\begin{equation}
\sum_{\tau=t}^{t+T-1}D_{nj}^d[\tau]=\sum_{\tau=t}^{t+T-1}S_{nj}^d+M_{nj}^d[t+T-1]-M_{nj}^d[t].\nonumber
\end{equation}
Moving terms to the left side and taking the absolute value of both sides,
\begin{align}
&\Bigl\vert\sum_{\tau=t}^{t+T-1}D_{nj}^d[\tau]-\sum_{\tau=t}^{t+T-1}S_{nj}^d\Bigr\vert\nonumber\\
=&\Bigl\vert M_{nj}^d[t+T-1]-M_{nj}^d[t]\Bigr|\nonumber\\
\leq&\bigl(\vert\cD\vert+1\bigr)c.\nonumber
\end{align}
As a result of sending out dummy packets in case of insufficient packets in $Q_{nj}^d[t]$, the departure process $D_{nj}^d[t]$ is deterministic given its associate token counts before time $t$. So we have
\begin{equation}
\bE\Bigl[\sum_{\tau=t}^{t+T-1}D_{nj}^d[\tau]\Bigm\vert\vM[t]\Bigr]=\sum_{\tau=t}^{t+T-1}D_{nj}^d[\tau].\nonumber
\end{equation}
Therefore
\begin{align}
&\biggl\vert\bE\Bigl[\sum_{\tau=t}^{t+T-1}D_{nj}^d[\tau]\Bigm\vert\vM[t]\Bigr]-\sum_{\tau=t}^{t+T-1}S_{nj}^d\biggr\vert\nonumber\\
\leq&\bigl(\vert\cD\vert+1\bigr)c.\label{eqn:smallrange}
\end{align}
Substituting~(\ref{eqn:arrivaldetailed}) into~(\ref{eqn:firsttermbound}), we have
\begin{align}
&\Delta V(\vQ[t])\nonumber\\
\leq&B+2\sum_d\sum_n\sum_jQ_{nj}^d[t]\cdot\nonumber\\
&\biggl\{\sum_{\tau=t}^{t+T-1}\bE\Bigl[X_n^d[\tau]\Bigr]\frac{\bar{R}_{nj}^{d}}{\sum_{m\in\cN}\bar{R}_{nm}^{d}}\mathbbm{1}_{\{n\in\cN_{source}^d\}}\nonumber\\
&+\sum_{\tau=t}^{t+T-1}\sum_{k\in\cN}\bE\Bigl[D_{kn}^d[\tau]\Bigm\vert \vM[t]\Bigr]\frac{\bar{R}_{nj}^{d}}{\sum_{m\in\cN}\bar{R}_{nm}^{d}}\nonumber\\
&-\sum_{\tau=t}^{t+T-1}\bE\Bigl[D_{nj}^d[\tau]\Bigm\vert \vM[t]\Bigr]\biggr\}\nonumber\\
=&B+2\sum_d\sum_n\sum_jQ_{nj}^d[t]\cdot\nonumber\\
&\biggl\{\sum_{\tau=t}^{t+T-1}\bE\Bigl[X_n^d[\tau]\Bigr]\frac{\bar{R}_{nj}^{d}}{\sum_{m\in\cN}\bar{R}_{nm}^{d}}\mathbbm{1}_{\{n\in\cN_{source}^d\}}\nonumber\\
&+\sum_{\tau=t}^{t+T-1}\sum_{k\in\cN}D_{kn}^d[\tau]\frac{\bar{R}_{nj}^{d}}{\sum_{m\in\cN}\bar{R}_{nm}^{d}}\nonumber\\
&-\sum_{\tau=t}^{t+T-1}\bE\Bigl[D_{nj}^d[\tau]\Bigm\vert \vM[t]\Bigr]\biggr\}\label{eqn:complete}
\end{align}
Applying~(\ref{eqn:smallrange}) and~(\ref{eqn:regulatorfixed}) to~(\ref{eqn:complete}) , then
\begin{align}
&(\ref{eqn:complete})\nonumber\\
\leq&B+2\sum_d\sum_n\sum_jQ_{nj}^d[t]\cdot\nonumber\\
&\biggl\{\sum_{\tau=t}^{t+T-1}\bE\Bigl[X_n^d[\tau]\Bigr]\frac{\bar{R}_{nj}^{d}}{\sum_{m\in\cN}\bar{R}_{nm}^{d}}\mathbbm{1}_{\{n\in\cN_{source}^d\}}\nonumber\\
&+\sum_{\tau=t}^{t+T-1}\sum_{k\in\cN}D_{kn}^d[\tau]\frac{\bar{R}_{nj}^{d}}{\sum_{m\in\cN}\bar{R}_{nm}^{d}}\nonumber\\
&-\Bigl(\sum_{\tau=t}^{t+T-1}S_{nj}^d-\bigl(\vert\cD\vert+1\bigr)c\Bigr)\biggr\}\nonumber\\
\leq&B+2\sum_d\sum_n\sum_jQ_{nj}^d[t]\biggl\{-T\delta+\bigl(\vert\cD\vert+1\bigr)c\biggl\}.\nonumber
\end{align}
Hence for any $T>\frac{\bigl(|\cD|+1\bigr)c}{\delta}$, the T step drift satisfies
\begin{equation}
\Delta V(\vQ[t])\leq B-2T\delta'\sum_d\sum_n\sum_jQ_{nj}^d[t],\label{eqn:secondtermbound}
\end{equation}
where $\delta'=T\delta-\bigl(|\cD|+1\bigr)c>0$. Applying Lemma 2 in~\cite{Neely03} to~(\ref{eqn:secondtermbound}), the proposition is proved.

\begin{biography}[{\includegraphics[width=1in,height=1.25in,clip,keepaspectratio]{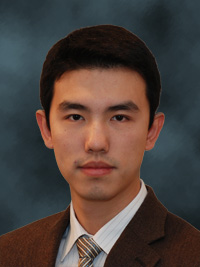}}]
  {Haozhi Xiong}(S '09) received his B.S. degree in Communication
  Engineering from Beijing University of Posts and Telecommunications,
  Beijing, China, in 2007, and the M.S. degree in Electrical and
  Computer Engineering from The Ohio State University in 2010. His
  research interests include optimal control in communication
  networks, low-delay scheduling algorithm and achitecture design, and
  control theory.
\end{biography}
\vspace{-0.5cm}
\begin{biography}[{\includegraphics[width=1in,height=1.25in,clip,keepaspectratio]{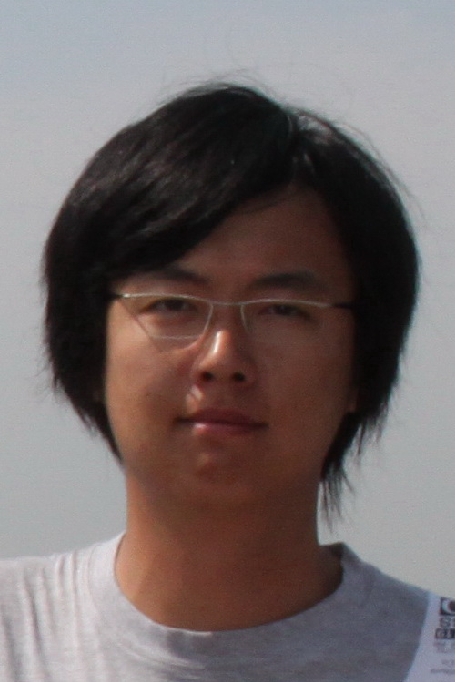}}]
{Ruogu Li}(S '10) received his B.S. degree in Electronic Engineering from Tsinghua University, Beijing, in 2007. He is currently a PhD student in Electrical and Computer Engineering at The Ohio State University. His research interests include optimal network control, wireless communication networks, low-delay scheduling scheme design and cross-layer algorithm design.
\end{biography}
\vspace{-0.5cm}
\begin{biography}[{\includegraphics[width=1in,height=1.25in,clip,keepaspectratio]{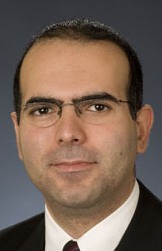}}]
  {Atilla Eryilmaz}(S '00-M '06) received his B.S. degree in
  Electrical and Electronics Engineering from Bo\=gazi\c ci
  University, Istanbul, in 1999, and the M.S. and Ph.D. degrees in
  Electrical and Computer Engineering from the University of Illinois
  at Urbana-Champaign in 2001 and 2005, respectively. Between 2005 and
  2007, he worked as a Postdoctoral Associate at the Laboratory for
  Information and Decision Systems at the Massachusetts Institute of
  Technology. Since 2007, he is an Assistant Professor of Electrical
  and Computer Engineering at The Ohio State University. He received
  the NSF-CAREER and the Lumley Research Awards in 2010, and his
  research interests include: design and analysis for communication
  networks, optimal control of stochastic networks, optimization
  theory, distributed algorithms, stochastic processes, and
  information theory.
\end{biography}
\vspace{-0.5cm}
\begin{biography}[{\includegraphics[width=1in,height=1.25in,clip,keepaspectratio]{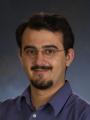}}] 
  {Eylem Ekici} received the BS and MS degrees in computer engineering
  from Bo\=gazi\c ci University, Istanbul, Turkey, in 1997 and 1998,
  respectively, and the PhD degree in electrical and computer
  engineering from Georgia Institute of Technology, Atlanta, in
  2002. Currently, he is an associate professor with the Department of
  Electrical and Computer Engineering, The Ohio State University. His
  current research interests include cognitive radio networks,
  nano-scale networks, vehicular communication systems, and wireless
  sensor networks, with a focus on routing and medium access control
  protocols, resource management, and analysis of network
  architectures and protocols. He is an associate editor of the
  IEEE/ACM Transactions on Networking, Computer Networks Journal
  (Elsevier), and the ACM Mobile Computing and Communications
  Review. He is a member of the IEEE.
\end{biography}
\end{document}